\documentclass[reqno]{amsart}

\usepackage{amssymb}
\usepackage{graphicx}
\usepackage[usenames, dvipsnames]{color}
\usepackage{verbatim}
\usepackage{mathrsfs}
\usepackage{esint}
\usepackage{empheq}
\usepackage[colorlinks,linkcolor=blue,citecolor=red]{hyperref}
\usepackage{enumerate}
\usepackage{enumitem}
\usepackage{indentfirst}
\usepackage{cases}
\usepackage{subfigure}

\numberwithin{equation}{section}

\newtheorem{theorem}{Theorem}[section]

\newtheorem{lemma}[theorem]{Lemma}
\newtheorem{prop}[theorem]{Proposition}

\theoremstyle{definition}
\newtheorem{example}{Example}
\theoremstyle{definition}
\newtheorem{remark}[theorem]{Remark}
\theoremstyle{definition}
\newtheorem{definition}[theorem]{Definition}
\theoremstyle{definition}

\newtheorem*{propA}{Proposition A}
\newtheorem*{propB}{Proposition B}

\makeatletter
\def\dashint{\operatorname%
{\,\,\text{\bf-}\kern-.98em\DOTSI\intop\ilimits@\!\!}}
\makeatother

\begin{document}
\title[Solvability on the exterior problem for the Monge--Amp\`{e}re equation]
{Two Necessary and Sufficient Conditions to the Solvability of the Exterior Dirichlet Problem for the Monge--Amp\`{e}re Equation}

\author[C. Wang]{Cong Wang}
\address[C. Wang]{School of Statistics, University of International Business and Economics, Beijing 100029, China.}
\email{cong\_wang@uibe.edu.cn}

\author[J. Bao]{Jiguang Bao}
\address[J. Bao]{School of Mathematical Sciences, Beijing Normal University, Laboratory of Mathematics and Complex Systems, Ministry of Education, Beijing 100875, China.}
\email{jgbao@bnu.edu.cn}
\thanks{J. Bao was supported by the National Key Research and Development Program of China (No. 2020YFA0712904)}




\begin{abstract}
The present paper provides two necessary and sufficient conditions for the existence of solutions to the exterior Dirichlet problem of the Monge--Amp\`{e}re equation with prescribed asymptotic behavior at infinity. By an adapted smooth approximation argument, we prove that the problem is solvable if and only if the boundary value is semi-convex with respect to the inner boundary, which is our first proposed new concept. Along the lines of Perron’s method for Laplace equation, we obtain the threshold for solvability in the asymptotic behavior at infinity of the solution, and remove the $C^2$ regularity assumptions on the boundary value and on the inner boundary which are required in the proofs of the corresponding existence theorems in the recent literatures.
\end{abstract}
\maketitle

\section{Introduction}

Let $\Omega$ be a bounded domain of $\mathbb{R}^n$, $n\geq3$, and let $\varphi$ be a function on $\partial\Omega$. In this paper, we intend to explore the solvability of the exterior Dirichlet problem for the Monge--Amp\`{e}re equation
\begin{equation}\label{eq:exteriorD}
  \begin{cases}
  \det(D^2u)=1\quad  \text{in}~\mathbb{R}^n\setminus\overline{\Omega},\\
  u=\varphi\qquad\qquad\quad~~~~  \text{on}~\partial\Omega,
  \end{cases}
\end{equation}
provided that $\Omega$ and $\varphi$ satisfy some general conditions, without the $C^2$ regularity like in Caffarelli--Li's work \cite{Caffarelli-Li-2003}.

The prototypical place where Monge--Amp\`{e}re equations arise is the Minkowski problem (see \cite{Nirenberg1953,Pogorelov1978}). Monge--Amp\`{e}re equations also play a significant role in the studies of affine geometry (see \cite{Pogorelov1972,Cheng-Yau1986,Trudinger-Wang2008}) and optimal transportation (see \cite{Philippis-Figalli2014}). The interior Dirichlet problem for Monge--Amp\`{e}re equations
\begin{equation}\label{eq:InterDiri}
  \begin{cases}
  \det(D^2u)=f & \text{in}~\Omega,\\
  u=\varphi & \text{on}~\partial\Omega,\\
  \end{cases}
\end{equation}
has a long history, and there have been many excellent results, especially on the solvability in different situations. We list several known existence results for solutions of \eqref{eq:InterDiri}, under the assumptions that $f\in C^\infty(\overline{\Omega})$, $f>0$ on $\overline{\Omega}$, and $\Omega$ is a bounded and convex domain with $\partial\Omega$ containing no line segment. Rauch--Taylor \cite{Rauch-Taylor1977} proved that \eqref{eq:InterDiri} has a unique convex solution $u\in C^0(\overline{\Omega})$ by Perron's method, when $\varphi\in C^0(\partial\Omega)$; see also Aleksandrov \cite{Aleksandrov1958} and Cheng--Yau \cite{ChengYau1977}. Pogorelov \cite{Pogorelov1971} obtained that the unique convex solution of \eqref{eq:InterDiri} is smooth in $\Omega$, when $\partial\Omega\in C^2$ and $\varphi\in C^2(\partial\Omega)$. After that, Caffarelli--Nirenberg--Spruck \cite{CNS1984CPAM} further proved that \eqref{eq:InterDiri} has a unique convex solution $u\in C^\infty(\overline{\Omega})$ by the continuity method, when $\Omega$ is strictly convex with $\partial\Omega\in C^\infty$ and $\varphi\in C^\infty(\partial\Omega)$; see also Krylov \cite{Krylov1983/85}.

By contrast, less results are known for the exterior Dirichlet problem for Monge--Amp\`{e}re equations. Differently from the interior Dirichlet problem \eqref{eq:InterDiri}, in exterior domains we also require the solutions to satisfy appropriate prescribed asymptotic behavior at infinity in order to restore the well-posedness. Such asymptotic behavior comes from Liouville-type theorems. The celebrated J\"{o}rgens--Calabi--Pogorelov theorem \cite{Jorgens1954,Calabi1958,Pogorelov1972} states that any classical convex solution of
\begin{equation}\label{eq0:MA=1}
\det(D^2u)=1
\end{equation}
in $\mathbb{R}^n$ must be a quadratic polynomial.
Caffarelli \cite{Caffarelli1995} generalized this result to viscosity solutions case. Caffarelli--Li \cite{Caffarelli-Li-2003} extended the J\"{o}rgens--Calabi--Pogorelov theorem to exterior domains. Specifically, they showed that if $n\geq3$ and $u$ is a viscosity solution of \eqref{eq0:MA=1} outside a bounded set, then there exist $A\in\mathcal{A}$, $b\in\mathbb{R}^n$ and $c\in\mathbb{R}$ such that
\begin{equation*}
  \lim_{|x|\to\infty}\bigg(u(x)-\bigg(\frac{1}{2}xAx^T+b\cdot x+c\bigg)\bigg)=0,
\end{equation*}
where
$$\mathcal{A}=\{A|\ A \text{ is a real }n\times n\text{ symmetric positive definite matrix with }\det(A)=1\}.$$
We refer to \cite{BCGJ-AMJ,LRW-JFA,WangBao-NLA} for more information about the Liouville-type theorems for Hessian equations.

Based on the asymptotic behavior above, Caffarelli--Li \cite{Caffarelli-Li-2003} proposed the following exterior Dirichlet problem
\begin{equation}\label{eq:DiriProb}
\begin{cases}
\det(D^2u)=1\quad\text{in}~\mathbb{R}^n\setminus\overline{\Omega},\\
u=\varphi\quad\text{on}~\partial\Omega,\\
\lim_{|x|\to\infty}\Big(u(x)-\Big(\frac{1}{2}xAx^T+b\cdot x+c\Big)\Big)=0.
\end{cases}
\end{equation}
Under the conditions of $\partial\Omega\in C^2$ and $\varphi\in C^2(\partial\Omega)$, Caffarelli--Li \cite{Caffarelli-Li-2003} proved the existence result by an adapted Perron's method, and then Li--Lu \cite{Li-Lu} gave the nonexistence result in terms of the asymptotic behavior. Therefore the characterization of solvability of \eqref{eq:DiriProb} is completed.

\begin{theorem}[Caffarelli--Li \cite{Caffarelli-Li-2003} and Li--Lu \cite{Li-Lu}]\label{thm:CL2003}
Let $\Omega$ be a bounded, strictly convex domain of $\mathbb{R}^n$, $n\geq3$, $\partial\Omega\in C^2$ and let $\varphi\in C^2(\partial\Omega)$. Then for any $A\in\mathbb{\mathcal{A}}$ and $b\in\mathbb{R}^n$, there exists some constant $c_*$, such that \eqref{eq:DiriProb} has a viscosity solution in $C^0(\mathbb{R}^n\setminus\Omega)$ if and only if $c\geq c_*$, where $c_*$ depends only on $n$, $A$, $b$, $\Omega$ and $\varphi$. 
\end{theorem}

Moreover, when the problem \eqref{eq:DiriProb} has a viscosity solution, it is unique by the comparison principle, and interior smooth by \cite{Caffarelli1995}.

Under the same assumptions on the domain and on the boundary value as in Theorem \ref{thm:CL2003}, the solvability of the exterior Dirichlet problem was also exploited for $k$-Hessian equations \cite{Bao-Li-Li-TAMS}, Hessian quotient equations \cite{Li-Li2018}, special Lagrangian equation \cite{Li2019TAMS}, and general Hessian-type equations \cite{Li-Bao2014,Jiang-Li-Li2021}. The ideas therein are similar in spirit to \cite{Caffarelli-Li-2003}.

The aim of this paper is to improve the $C^2$ regularity condition of both domains and boundary values in Theorem \ref{thm:CL2003}. We are interested in studying the solvability of \eqref{eq:DiriProb} under the geometry and regularity conditions which are corresponding to that of Laplace equation. For the interior Dirichlet problem for Laplace equation, to guarantee the continuity up to the boundary of the solution, the continuous boundary value is necessary, and the domain needs to be regular (i.e. there exists a barrier function at each boundary point); see \cite[Chapter 2.8]{GT}. A known sufficient condition that makes the domain be regular is the exterior sphere condition. For the exterior Dirichlet problem for Laplace equation with prescribed limit at infinity, there is a unique solution when $n\geq3$, the domain is smooth and the boundary is continuous; see for instance Meyers--Serrin \cite{Meyers-Serrin}.

We firstly focus on the boundary value $\varphi$ on $\partial\Omega$. For interior Dirichlet problem \eqref{eq:InterDiri}, the solution exists if and only if $\varphi\in C^{0}(\partial\Omega)$ due to Rauch--Taylor \cite{Rauch-Taylor1977}. While for the exterior Dirichlet problem \eqref{eq:DiriProb},
the convexity of $\mathbb{R}^n\setminus\overline\Omega$ is opposite to that of $\Omega$,
which may lead to that $\varphi$ in \eqref{eq:DiriProb} has different and even stronger structure from $\varphi\in C^0(\partial\Omega)$ in \eqref{eq:InterDiri}. Based on such observation, we are naturally motivated to investigate the necessary condition of $\varphi$ for the existence of solutions. Unexpectedly, we derive that the ``semi-convexity'' condition below, which is a different phenomenon from interior Dirichlet problem \eqref{eq:InterDiri}.

In order to introduce the new concept of ``semi-convexity'' clearly, we first introduce the local coordinate system at a boundary point. For any fixed $\xi\in\partial\Omega$, we choose a coordinate system $(x',x_n)$ such that $x=0$ at $\xi$, the positive $x_n$-axis directs to the interior of $\Omega$,  and $\partial\Omega$ can be locally represented by the graph of
\begin{equation}\label{eq:bdyre}
x_n=\rho(x')\quad\text{for}~|x'|<\delta(\xi)
\end{equation}
for some constant $\delta(\xi)>0$, where $\rho:B'_{\delta(\xi)}(0')=\{x'\in\mathbb{R}^{n-1}|\,|x'|<\delta(\xi)\}\to\mathbb{R}$ is a function with $\rho(0')=0$.
We call such coordinate system the local coordinate system at $\xi$.

\begin{definition}\label{defn:semi-convex}
  Let $\varphi$ be a function defined on $\partial\Omega$. We say that $\varphi$ is semi-convex with respect to $\partial\Omega$ at $\xi$, if under the local coordinate system at $\xi$,
  the function
  $$\psi(x'):=\varphi(x',\rho(x'))$$
  is semi-convex in $B'_{\delta(\xi)}(0')$, that is, there exists a constant $K(\xi)>0$, such that $\psi(x')+\frac{K(\xi)}{2}|x'|^2$ is convex in $B'_{\delta(\xi)}(0')$. We say that $\varphi$ is semi-convex with respect to $\partial\Omega$, if $\varphi$ is semi-convex with respect to $\partial\Omega$ at each $\xi\in\partial\Omega$.
\end{definition}

Our first main result is that the boundary value is necessarily semi-convex with respect to the boundary for the existence of solutions of \eqref{eq:DiriProb}.
\begin{theorem}\label{thm:necessity}
Let $\Omega$ be a bounded convex domain of $\mathbb{R}^n$, $n\geq3$, $\partial\Omega\in C^3$. Let $u\in C^0(\mathbb{R}^n\setminus\Omega)$ be a viscosity solution of \eqref{eq:exteriorD}, then $\varphi$ is semi-convex with respect to $\partial\Omega$.
\end{theorem}

With the necessity in hand, we are inspired to study the solvability of \eqref{eq:DiriProb} under the semi-convexity condition. We further focus on the domain $\Omega$.
Following \cite[Chapter 14.2]{GT}, we use the geometry concept of enclosing sphere condition below, which was also used in Urbas's work on studying prescribed Gauss curvature problem \cite{Urbas-1984,Urbas-1998}. Such geometry condition on the domain is much weaker than $C^2$ regularity condition.

\begin{definition}\label{defn:exterior-shpere}
  We say that $\Omega$ satisfies an enclosing sphere condition at $\xi\in\partial\Omega$, if there exists a ball $B=B_{r(\xi)}(y(\xi))\supset\Omega$ satisfying $\xi\in\partial\Omega\cap\partial B$. We say that $\Omega$ satisfies an enclosing sphere condition, if $\Omega$ satisfies an enclosing sphere condition at each $\xi\in\partial\Omega$. Moreover, we say that $\Omega$ satisfies a uniform enclosing sphere condition, if
  $r(\xi)$ is bounded on $\partial\Omega$.
\end{definition}

Suppose that the semi-convexity with respect to the boundary and a uniform enclosing sphere conditions hold. Our second main result is a necessary and sufficient condition to the existence of solution of \eqref{eq:DiriProb} in terms of the asymptotic behavior near infinity.
\begin{theorem}\label{thm:main}
  Let $\Omega$ be a domain of $\mathbb{R}^n$ satisfying a uniform enclosing sphere condition,
  $n\geq3$, $\partial\Omega\in C^1$. Let $\varphi$ be semi-convex with respect to $\partial\Omega$. Then for any $A\in\mathcal{A}$ and $b\in\mathbb{R}^n$, there exists some constant $c_*$, such that \eqref{eq:DiriProb} has a viscosity solution in $C^0(\mathbb{R}^n\setminus\Omega)$ if and only if $c\geq c_*$, where $c_*$ depends only on $n$, $A$, $b$, $\Omega$ and $\varphi$.
\end{theorem}

If $\Omega$ is a bounded and strictly convex domain with $\partial\Omega\in C^2$, then $\Omega$ satisfies a uniform enclosing sphere condition 
(see Proposition \hyperlink{A}{A} in Appendix). The strict convexity of a $C^2$ domain $\Omega$ throughout the paper refers to Caffarelli--Li \cite{Caffarelli-Li-2003}. Namely, principal curvatures of $\partial\Omega$ are positive.
Also, if $\varphi\in C^2(\partial\Omega)$, then direct calculation shows that $\varphi$ is semi-convex with respect to $\partial\Omega$. Therefore, Theorem \ref{thm:main} is more general than Theorem \ref{thm:CL2003} of Caffarelli--Li \cite{Caffarelli-Li-2003} and Li--Lu \cite{Li-Lu}.

By combining Theorem \ref{thm:necessity}, we conclude the following necessary and sufficient condition to the existence
of solution of \eqref{eq:DiriProb} in terms of the boundary value.

\begin{theorem}\label{thm:NS2}
Let $\Omega$ be a bounded, strictly convex domain of $\mathbb{R}^n$, $n\geq3$, $\partial\Omega\in C^3$. Then for any $A\in\mathcal{A}$, $b\in\mathbb{R}^n$, there exists a constant $c$ such that \eqref{eq:DiriProb} has a viscosity solution in $C^0(\mathbb{R}^n\setminus\Omega)$ if and only if $\varphi$ is semi-convex with respect to $\partial\Omega$.
\end{theorem}

We see from the above theorem that for the Monge--Amp\`{e}re equation, the semi-convex boundary value in the exterior Dirichlet problem is in the same position as the continuous boundary value in the interior Dirichlet problem.

We turn to point out that there exist many functions which satisfy the semi-convexity condition but $\varphi\notin C^2(\partial\Omega)$; see Example 1 in Appendix. Also, there exist many domains $\Omega$ which satisfy a uniform enclosing sphere condition and make Theorem \ref{thm:main} work but even not $C^1$; see Example 3. It would be interesting to see if Theorem \ref{thm:main} remains valid under weaker assumptions on $\partial\Omega$.
In addition, we mention that the ``uniform'' in a uniform enclosing sphere condition could not be dropped; see Example 2.

The constant $c_*$ in Theorem \ref{thm:main} and Theorem \ref{thm:CL2003} depends on different quantities of $\Omega$ and $\varphi$. Precisely, the former depends on $C^1$ regularity of $\partial\Omega$, the uniform radius of enclosing sphere of $\Omega$ and the semi-convexity of $\varphi$, while the latter depends on the diameter and the strict convexity of $\Omega$, the $C^2$ norm of $\partial\Omega$ and $\|\varphi\|_{C^2(\partial\Omega)}$.

We now comment the proof of Theorem \ref{thm:main}. The proof of existence part is based on Perron's method, and processes the equation, the boundary value and the asymptotic behavior in \eqref{eq:DiriProb} separately. Along the lines of Perron's method for Laplace equation, our proof is traditional and elementary. We first introduce the lifting function with respect to the Monge--Amp\`{e}re equation and prove that Perron's solution satisfies the Monge--Amp\`{e}re equation in the viscosity sense. We then verify that such solution also satisfies the asymptotic behavior at infinity when $c$ is sufficiently large. It is worth to note that Perron's solution satisfies the equation and the asymptotic behavior under fairly weak assumptions on domains and boundary values. To handle the interior boundary behavior, compared with \cite{Caffarelli-Li-2003}, key changes need to be made due to the weaker conditions by reproving the barrier lemma (see Lemma \ref{lem:barrier}). Once the existence part is established, we can follow almost without change as in \cite[Theorem 1.2]{Li-Lu} to obtain the nonexistence part.

The rest of this paper is organized as follows. In Section \ref{sec:necessity}, we establish Theorem \ref{thm:necessity} using an adapted smooth approximation argument. In Section \ref{sec:eq}, we prove that the Perron's solution is a viscosity solution of the Monge--Amp\`{e}re equation. In Section \ref{sec:asy}, we further check the Perron's solution satisfies the asymptotic quadratic behavior at infinity. In Section \ref{sec:bdy}, we investigate the boundary behavior of the Perron's solution. Section \ref{sec:main} is devoted to prove Theorem \ref{thm:main}.

We fix some notations throughout this paper. For $x=(x_1,\cdots,x_n)\in\mathbb{R}^n$, we denote $x'=(x_1,\cdots,x_{n-1})\in\mathbb{R}^{n-1}$. For $r>0$,
we denote by $B_r(x)$ the open ball in $\mathbb{R}^n$ centered at $x$ of radius $r$, and by $B'_r(x')$ the open ball in $\mathbb{R}^{n-1}$ centered at $x'$ of radius $r$. For a function $u(x)$, we denote by $u_{i}$ the partial derivative with respect to $x_i$, and by $u_{ij}$ the second derivative with respect to $x_i$ and $x_j$. We also use $Du$ and $D^2u$ to denote $(u_1,\cdots,u_{n})$ and the $n\times n$ matrix $(u_{ij})$, respectively. Similarly, for a function $v(x')$, we used
$D'v$ and $D'^2v$ to denote $(v_1,\cdots,v_{n-1})$ and the $(n-1)\times (n-1)$ matrix $(v_{ij})$, respectively.
\section{Proof of Theorem \ref{thm:necessity}}\label{sec:necessity}

For the reader's convenience, we recall that $u\in C^0(D)$ is locally convex in an open set $D\subset\mathbb{R}^n$ if for any $x\in D$, there is an open ball $B\subset D$ centered at $x$ such that the restriction of $u$ to $B$ is convex, which is equivalent to that $u$ is convex in any open ball $B\subset D$ (see for instance \cite{yan2014}).

In this section, we derive the necessity of semi-convexity condition of the boundary value in a very general setting through an adapted smooth approximation argument, which implies Theorem \ref{thm:necessity}.

\begin{prop}
  Let $\Omega$ be a bounded convex domain of $\mathbb{R}^n$, $n\geq3$, $\partial\Omega\in C^3$. Let $u\in C^{0}(\mathbb{R}^n\setminus\Omega)$ be locally convex, and $\varphi(x):=u(x)$ for $x\in\partial\Omega$. Then $\varphi$ is semi-convex with respect to $\partial\Omega$.
\end{prop}
\begin{proof}
For $\varepsilon>0$, denote $\Omega^\varepsilon=\{x\in\mathbb{R}^n|\,\text{dist}(x,\Omega)<\varepsilon\}$. Define for $0<\varepsilon<1$,
$$u^\varepsilon(x)=\int_{B_1(0)}\eta(z)u(x-\varepsilon z)\mathrm{d}z\quad\text{for}~x\in\mathbb{R}^n\setminus\overline{\Omega^\varepsilon},$$
where $\eta\in C^\infty(\mathbb{R}^n)$ is the standard mollifier given by
\begin{equation*}
\eta(x)=
\begin{cases}
C_0\exp{\Big(\frac{1}{|x|^2-1}\Big)} & \quad\text{if}~|x|<1,\\
0 & \quad\text{if}~|x|\geq1,
\end{cases}
\end{equation*}
with the constant $C_0$ satisfying $\int_{\mathbb{R}^n}\eta(x)\mathrm{d}x=1$. Clearly, $u^\varepsilon\in C^\infty(\mathbb{R}^n\setminus\overline{\Omega^\varepsilon})$. It can be verified that $u^\varepsilon$ is locally convex in $\mathbb{R}^n\setminus\overline{\Omega^\varepsilon}$. Indeed, given a ball $B_r(x_0)\subset\mathbb{R}^n\setminus\overline{\Omega^\varepsilon}$, $x,y\in B_r(x_0)$ and $z\in B_1(0)$, we have
$$x-\varepsilon z,~y-\varepsilon z\in B_{r+\varepsilon}(x_0)\subset\mathbb{R}^n\setminus\overline{\Omega}.$$
It follows from the convexity of $u$ in $B_{r+\varepsilon}(x_0)$ and $\eta\geq0$ that for $0<t<1$,
\begin{equation*}
\begin{split}
u^\varepsilon(tx+(1-t)y) & =\int_{B_1(0)}\eta(z)u(tx+(1-t)y-\varepsilon z)\mathrm{d}z \\
 & =\int_{B_1(0)}\eta(z)u(t(x-\varepsilon z)+(1-t)(y-\varepsilon z))\mathrm{d}z \\
 & \leq t\int_{B_1(0)}\eta(z)u(x-\varepsilon z)\mathrm{d}z+(1-t)\int_{B_1(0)}\eta(z)u(y-\varepsilon z)\mathrm{d}z \\
 & =tu^\varepsilon(x)+(1-t)u^\varepsilon(y).
\end{split}
\end{equation*}
Thus we have proved $u^\varepsilon$ is locally convex in $\mathbb{R}^n\setminus\overline{\Omega^\varepsilon}$.

Fix $\xi\in\partial\Omega$ and assume $\partial\Omega$ can be locally represented by the graph of
$$x_n=\rho(x')\geq0\quad\text{for}\,|x'|<\delta(\xi),$$
for some $\delta(\xi)>0$.
Denote $\nu(x)=(\nu'(x),\nu^{(n)}(x))$ the unit inner normal vector of $\partial\Omega$ at $x=(x',\rho(x'))$. Since $\Omega$ is convex, we have for $x=(x',\rho(x'))$,
$$x-2\varepsilon\nu(x)=(x'-2\varepsilon\nu'(x',\rho(x')),\rho(x')-2\varepsilon\nu^{(n)}(x',\rho(x')))\in\partial\Omega^{2\varepsilon}.$$
Define
\begin{equation*}
\psi^\varepsilon(x')=u^\varepsilon(x'-2\varepsilon\mu'(x'),\rho(x')-2\varepsilon\mu^{(n)}(x'))\quad\text{for}\,|x'|<\delta(\xi),
\end{equation*}
where $\mu'(x')=(\mu^{(1)}(x'),\cdots,\mu^{(n-1)}(x'))=\nu'(x',\rho(x'))$, $\mu^{(n)}(x')=\nu^{(n)}(x',\rho(x'))$. Direct calculation shows for $i,j=1,\cdots,n-1$,
$$\psi^\varepsilon_i=\sum_{k=1}^{n-1}u^\varepsilon_k\left(\delta_{ik}-2\varepsilon\mu^{(k)}_i\right)+u^\varepsilon_n\left(\rho_i-2\varepsilon\mu^{(n)}_i\right),$$
and
\begin{equation*}
\begin{split}
\psi^\varepsilon_{ij} = &\, \sum_{k=1}^{n-1}\left(\left(\sum_{l=1}^{n-1}u^\varepsilon_{kl}\left(\delta_{jl}-2\varepsilon\mu^{(l)}_j\right)+u^\varepsilon_{kn}\left(\rho_j-2\varepsilon\mu^{(n)}_j\right)\right)
\left(\delta_{ik}-2\varepsilon\mu^{(k)}_i\right)+u^\varepsilon_k\left(-2\varepsilon\mu^{(k)}_{ij}\right)\right) \\
&\,+\left(\sum_{k=1}^{n-1}u^\varepsilon_{nk}\left(\delta_{kj}-2\varepsilon\mu^{(k)}_j\right)+u^\varepsilon_{nn}\left(\rho_j-2\varepsilon\mu^{(n)}_j\right)
\right)\left(\rho_i-2\varepsilon\mu^{(n)}_i\right)+u^\varepsilon_n\left(\rho_{ij}-2\varepsilon\mu^{(n)}_{ij}\right) \\
 = &\,\sum_{k,l=1}^{n-1}u^\varepsilon_{kl}\left(\delta_{ik}-2\varepsilon\mu^{(k)}_i\right)\left(\delta_{jl}-2\varepsilon\mu^{(l)}_j\right)\\
 &\,+\sum_{k=1}^{n-1}u^\varepsilon_{kn}\left(\delta_{ik}-2\varepsilon\mu^{(k)}_i\right)\left(\rho_j-2\varepsilon\mu^{(n)}_j\right)
 +\sum_{k=1}^{n-1}u^\varepsilon_{nk}\left(\delta_{kj}-2\varepsilon\mu^{(k)}_j\right)\left(\rho_i-2\varepsilon\mu^{(n)}_i\right) \\
 &\,+ u^\varepsilon_{nn}\left(\rho_i-2\varepsilon\mu^{(n)}_i\right)\left(\rho_j-2\varepsilon\mu^{(n)}_j\right)+u^\varepsilon_n\rho_{ij}
 -2\varepsilon\Bigg(\sum_{k=1}^{n-1}u^\varepsilon_k\mu^{(k)}_{ij}+u^\varepsilon_n\mu^{(n)}_{ij}\Bigg),
\end{split}
\end{equation*}
where $\delta_{ij}=1$ if $i=j$ and $\delta_{ij}=0$ if $i\neq j$. For $\zeta\in\mathbb{R}^{n-1}$, it follows that

\begin{equation*}
\begin{split}
\sum_{i,j=1}^{n-1}\psi^\varepsilon_{ij}\zeta_i\zeta_j= &\,\sum_{k,l=1}^{n-1}u^\varepsilon_{kl}\Bigg(\sum_{i=1}^{n-1}\left(\delta_{ik}-2\varepsilon\mu^{(k)}_i\right)\zeta_i\Bigg)
\Bigg(\sum_{j=1}^{n-1}\left(\delta_{jl}-2\varepsilon\mu^{(l)}_j\right)\zeta_j\Bigg) \\
&\,+2\sum_{k=1}^{n-1}u^\varepsilon_{kn}\Bigg(\sum_{i=1}^{n-1}\left(\delta_{ik}-2\varepsilon\mu^{(k)}_i\right)\zeta_i\Bigg)
\Bigg(\sum_{j=1}^{n-1}\left(\rho_j-2\varepsilon\mu^{(n)}_i\right)\zeta_j\Bigg)\\
 &\,+ u^\varepsilon_{nn}\Bigg(\sum_{i=1}^{n-1}\left(\rho_i-2\varepsilon\mu^{(n)}_i\right)\zeta_i\Bigg)
 \Bigg(\sum_{j=1}^{n-1}\left(\rho_j-2\varepsilon\mu^{(n)}_j\right)\zeta_j\Bigg) \\
 &\,+u^\varepsilon_n\sum_{i,j=1}^{n-1}\rho_{ij}\zeta_i\zeta_j
 -2\varepsilon\Bigg(\sum_{k=1}^{n-1}u^\varepsilon_k\sum_{i,j=1}^{n-1}\mu^{(k)}_{ij}\zeta_i\zeta_j+u^\varepsilon_n\sum_{i,j=1}^{n-1}\mu^{(n)}_{ij}\zeta_i\zeta_j\Bigg).
\end{split}
\end{equation*}
Since $D^2u^\varepsilon\geq0$, $D'^2\rho\geq0$ and $(\mu^{(1)}_{ij}),\cdots,(\mu^{(n)}_{ij})$ are bounded due to $\partial\Omega\in C^3$,
\begin{equation}\label{eq:psi-e}
\begin{split}
\sum_{i,j=1}^{n-1}\psi^\varepsilon_{ij}\zeta_i\zeta_j
 &\,\geq u^\varepsilon_n\sum_{i,j=1}^{n-1}\rho_{ij}\zeta_i\zeta_j
 -2\varepsilon\Bigg(\sum_{k=1}^{n-1}u^\varepsilon_k\sum_{i,j=1}^{n-1}\mu^{(k)}_{ij}\zeta_i\zeta_j+u^\varepsilon_n\sum_{i,j=1}^{n-1}\mu^{(n)}_{ij}\zeta_i\zeta_j\Bigg) \\
 &\,\geq C\left(\min\{\min_{\Gamma^{2\varepsilon}}u^\varepsilon_n,0\}-2\max_{\Gamma^{2\varepsilon}}\varepsilon|Du^\varepsilon|\right)|\zeta|^2,
\end{split}
\end{equation}
where $C>0$ is a constant depending only on $n$ and the $C^3$ norm of $\partial\Omega$, and
$$\Gamma^{2\varepsilon}=\{x-2\varepsilon v(x)|\,x=(x',\rho(x')),~|x'|\leq\delta(\xi)\}\subset\partial\Omega^{2\varepsilon}.$$

To prove that $\psi^\varepsilon$ is semi-convex in $B'_{\delta(\xi)}(0')$, we first estimate $u_n^\varepsilon$ on $\Gamma^{2\varepsilon}$ from below. Take $R>0$ such that $\Omega^3\subset B_R(0)$. For $x_0\in\Gamma^{2\varepsilon}$, we denote by $y_0$ the intersection of $l$ and $\partial B_{R}(0)$, where $l$ is the ray along the direction of negative $x_n$-axis with the endpoint $x_0$. Then the open line segment $\overline{x_0y_0}$ connecting $x_0$ and $y_0$ is contained in $\mathbb{R}^n\setminus\overline{\Omega^{\varepsilon}}$. Hence for
$$\hat{x}=\frac{|y_0-\hat{x}|}{|y_0-x_0|}x_0+\frac{|\hat{x}-x_0|}{y_0-x_0}y_0\in\overline{x_0y_0},$$
we have by the convexity of $u^\varepsilon$,
$$u^\varepsilon(\hat{x})\leq\frac{|y_0-\hat{x}|}{|y_0-x_0|}u^\varepsilon(x_0)+\frac{|\hat{x}-x_0|}{|y_0-x_0|}u^\varepsilon(y_0).$$
Together with $|y_0-x_0|\geq1$, we get
$$\frac{u^\varepsilon(\hat{x})-u^\varepsilon(x_0)}{|\hat{x}-x_0|}\leq\frac{u^\varepsilon(y_0)-u^\varepsilon(x_0)}{|y_0-x_0|}
\leq2\max_{\overline{B_{R}(0)\setminus\Omega^{2\varepsilon}}}|u^\varepsilon|.$$
The boundedness of $u$ on $\overline{B_{R+1}(0)\setminus\Omega}$ yields that
\begin{equation*}
|u^\varepsilon(x)|\leq\int_{B_1(0)}\eta(z)|u(x-\varepsilon z)|\mathrm{d}z\leq\max_{z\in\overline{B_1(0)}}|u(x-\varepsilon z)|
\end{equation*}
holds uniformly for $x\in\overline{B_R(0)\setminus\Omega^{2\varepsilon}}$. Then
$$\max_{\overline{B_{R}(0)\setminus\Omega^{2\varepsilon}}}|u^\varepsilon|
\leq\max_{\overline{B_{R+1}(0)\setminus\Omega}}|u|.$$
Hence,
$$\frac{u^\varepsilon(\hat{x})-u^\varepsilon(x_0)}{|\hat{x}-x_0|}
\leq2\max_{\overline{B_{R+1}(0)\setminus\Omega}}|u|.$$
That is,
\begin{equation}\label{infun-e}
\min_{\partial\Omega^{2\varepsilon}}u^\varepsilon_n\geq-2\max_{\overline{B_{R+1}(0)\setminus\Omega}}|u|.
\end{equation}

We continue to estimate $\varepsilon|Du^\varepsilon|$ on $\partial\Omega^{2\varepsilon}$ from above. It is clear that
$$u^\varepsilon(x)=\int_{\mathbb{R}^n\setminus\overline{\Omega}}\frac{1}{\varepsilon^n}\eta\bigg(\frac{x-z}{\varepsilon}\bigg)u(z)\mathrm{d}z
\quad\text{for}~x\in\mathbb{R}^n\setminus\overline{\Omega^\varepsilon}.$$
Then direct calculation gives
\begin{equation*}
\begin{split}
|u^\varepsilon_i(x)| &=\bigg|\int_{\mathbb{R}^n\setminus\overline{\Omega}}\frac{1}{\varepsilon^n}\frac{\partial\eta}{\partial x_i}\bigg(\frac{x-z}{\varepsilon}\bigg)u(z)\mathrm{d}z\bigg| \\
&=\frac{1}{\varepsilon}\bigg|\int_{B_1(0)}\eta_i(z)u(x-\varepsilon z)\mathrm{d}z\bigg| \\
&\leq\frac{1}{\varepsilon}\max_{z\in\overline{B_{1}(0)}}|u(x-\varepsilon z)|\int_{B_1(0)}|D\eta(z)|\mathrm{d}z \\
&=:\frac{C(n)}{\varepsilon}\max_{z\in\overline{B_{1}(0)}}|u(x-\varepsilon z)|.
\end{split}
\end{equation*}
Using the boundedness of $u$ in $\overline{\Omega^{3}\setminus\Omega}$, we obtain for $x_0\in\partial\Omega^{2\varepsilon}$,
$$|u^\varepsilon_i(x_0)|\leq\frac{C(n)}{\varepsilon}\max_{\overline{\Omega^{3}\setminus\Omega}}|u|.$$
That is,
\begin{equation}\label{eq:Du-e}
\max_{\partial\Omega^{2\varepsilon}}\varepsilon|Du^\varepsilon|\leq C(n)\max_{\overline{\Omega^{3}\setminus\Omega}}|u|.
\end{equation}

Combining \eqref{eq:psi-e} with \eqref{infun-e} and \eqref{eq:Du-e}, we obtain that there exists a constant $K(\xi)>0$ independent of $\varepsilon$ such that
$$D^2\psi^\varepsilon(x')\geq-KI_{n-1}\quad\text{for}~x'\in B'_{\delta(\xi)}(0'),$$
where $I_{n-1}$ is the $(n-1)\times(n-1)$ identity matrix. That is $\psi^\varepsilon(x')+\frac{K(\xi)}{2}|x'|^2$ is convex in $B'_{\delta(\xi)}(0')$. Then for $x',y'\in B'_{\delta(\xi)}(0')$ and $0<t<1$,
\begin{equation}\label{eq:psi-e-s}
\begin{split}
&\, \psi^\varepsilon(tx'+(1-t)y')+\frac{K(\xi)}{2}|tx'+(1-t)y'|^2 \\
\leq &\, t\bigg(\psi^\varepsilon(x')+\frac{K(\xi)}{2}|x'|^2\bigg)+(1-t)\bigg(\psi^\varepsilon(y')+\frac{K(\xi)}{2}|y'|^2\bigg).
\end{split}
\end{equation}
From the uniform continuity of $u$ in $\overline{\Omega^{3}\setminus\Omega}$, it follows that for $\psi(x'):=\varphi(x',\rho(x'))$,
\begin{equation*}
\begin{split}
 |\psi^\varepsilon(x')-\psi(x')|
= &\, |u^\varepsilon(x-2\varepsilon\nu(x))-u(x)| \\
\leq &\, \int_{B_1(0)}\eta(z)\left|u(x-2\varepsilon\nu(x)-\varepsilon z)-u(x)\right|\mathrm{d}z \\
\leq &\, \max_{z\in\overline{B_{1}(0)}}\left|u(x-2\varepsilon\nu(x)-\varepsilon z)-u(x)\right| \\
\to &\,0\quad\quad\text{as}~\varepsilon\to0
\end{split}
\end{equation*}
holds uniformly for $x=(x',\rho(x'))$ with $|x'|<\delta(\xi)$. By sending $\varepsilon\to0$ in \eqref{eq:psi-e-s}, we obtain for $x',y'\in B'_{\delta(\xi)}(0')$ and $0<t<1$,
\begin{equation*}
\begin{split}
&\, \psi(tx'+(1-t)y')+\frac{K(\xi)}{2}|tx'+(1-t)y'|^2 \\
\leq &\, t\bigg(\psi(x')+\frac{K(\xi)}{2}|x'|^2\bigg)+(1-t)\bigg(\psi(y')+\frac{K(\xi)}{2}|y'|^2\bigg).
\end{split}
\end{equation*}
Hence, $\psi$ is semi-convex in $B'_{\delta(\xi)}(0')$. Since $\xi\in\partial\Omega$ is arbitrary, the proof is completed.
\end{proof}

\section{Perron's solution of equation \eqref{eq0:MA=1} in the exterior domain}\label{sec:eq}

In this section, we give the expression of the Perron's solution of the exterior Dirichlet problem \eqref{eq:DiriProb}, under fairly general assumptions that  the domain $\Omega$ and the boundary value $\varphi$ are bounded. We will verify that the Perron's solution is indeed a viscosity solution of equation \eqref{eq0:MA=1} in the exterior domain $\mathbb{R}^n\setminus\overline{\Omega}$. The proof is similar in spirit to that of Laplace equation.

For the reader's convenience, we recall the definition of viscosity solutions. Let $D$ be an open set in $\mathbb{R}^n$ and let $u\in C^0(D)$ be a locally convex function. We say that $u$ is a viscosity subsolution of \eqref{eq0:MA=1} in $D$, if for any function $v\in C^2(D)$ and any local maximum point $x_0\in D$ of $u-v$,  we have
$$\det(D^2v(x_0))\geq1.$$
We say that $u$ is a viscosity supersolution of \eqref{eq0:MA=1} in $D$, if for any local convex function $v\in C^2(D)$ and any local minimum point $x_0\in D$ of $u-v$, we have
$$\det(D^2v(x_0))\leq1.$$
$u$ is a viscosity solution of \eqref{eq0:MA=1}, if it is both a viscosity subsolution and a viscosity supersolution of \eqref{eq0:MA=1}. It is clear from the definition that if $u$ is a viscosity solution of \eqref{eq0:MA=1} in each ball $B\subset\subset D$, then $u$ is a viscosity solution of \eqref{eq0:MA=1} in $D$.

To begin with, we introduce the definition of the lifting function with respect to the Monge--Amp\`{e}re equation.

\begin{definition}\label{defn:lifting}
Let $B\subset\subset D$ be an open ball, and $u\in C^0(D)$.
We define the lifting function of $u$ with respect to the Monge--Amp\`{e}re equation \eqref{eq0:MA=1} in $B$ by
\begin{equation*}
U(x)=
\begin{cases}
  \overline{u}(x) &\quad\text{if}~x\in B, \\
  u(x) &\quad\text{if}~x\in D\setminus B,
\end{cases}
\end{equation*}
where $\overline{u}\in C^0(\overline{B})$ is convex in $B$ solving
\begin{equation*}
\begin{cases}
\det(D^2\overline{u})=1 & \text{in }B,\\
\overline{u}=u & \text{on }\partial B.
\end{cases}
\end{equation*}
The existence of $\overline{u}$ is guaranteed by \cite[Theorem 4.1]{Rauch-Taylor1977} (see also \cite[Lemma A.3]{Caffarelli-Li-2003} for a direct proof).
\end{definition}

\begin{prop}\label{prop:lifting}
Let $u$ be a viscosity subsolution of \eqref{eq0:MA=1} in $D$, and let $U$ be the lifting function of $u$ with respect to \eqref{eq0:MA=1} in $B$. Then $U\geq u$ in $D$, and $U$ is also a viscosity subsolution of \eqref{eq0:MA=1} in $D$.
\end{prop}

\begin{proof}
  The comparison principle gives $\overline{u}\geq u$ in $B$, and so $U\geq u$ in $D$.

  We claim that $U$ is locally convex in $D$. Indeed, it suffices to prove that $U$ is convex in an open ball $B_0\subset D$ centered at $x_0\in\partial B$. For any fixed $x,y\in B_0$, we only need to consider the case of $x\in B_0\cap B$ and $y\in B_0\setminus B$, as otherwise the conclusion follows immediately.

  \noindent
  \textbf{Case 1.} For any $0<t<1$ such that $tx+(1-t)y\in B_0\setminus B$, by the convexity of $u$ in $B_0$, we have
  \begin{equation*}
    \begin{split}
       U(tx+(1-t)y) & =u(tx+(1-t)y) \\
         & \leq tu(x)+(1-t)u(y) \\
         & \leq tU(x)+(1-t)U(y).
    \end{split}
  \end{equation*}
  \textbf{Case 2.} For any $0<t<1$ such that $z:=tx+(1-t)y\in B_0\cap B$, let $\hat{z}=L_{xy}\cap\partial B$, where $L_{xy}$ is the line segment connecting $x$ and $y$. Then
  $z=\hat{t}x+(1-\hat{t})\hat{z}$ for some $0<\hat{t}<1$. By $|x-\hat{z}|\geq|z-\hat{z}|$, we have
  $$t=\frac{|z-y|}{|x-y|}=\frac{|z-\hat{z}|+|\hat{z}-y|}{|x-\hat{z}|+|\hat{z}-y|}\geq\frac{|z-\hat{z}|}{|x-\hat{z}|}=\hat{t}.$$
  Since $\overline{u}$ is convex in $B$ and $\overline{u}\geq u$ in $B$, it follows that
  \begin{equation*}
    \begin{split}
       U(tx+(1-t)y) & =\overline{u}(z) \\
         & \leq\hat{t}\overline{u}(x)+(1-\hat{t})\overline{u}(\hat{z}) \\
         & =\hat{t}\overline{u}(x)+(1-\hat{t})u\bigg(\frac{t-\hat{t}}{1-\hat{t}}x+\frac{1-t}{1-\hat{t}}y\bigg) \\
         & \leq\hat{t}\overline{u}(x)+(t-\hat{t})u(x)+(1-t)u(y) \\
         & \leq t\overline{u}(x)+(1-t)u(y) \\
         & \leq tU(x)+(1-t)U(y).
    \end{split}
  \end{equation*}

  By Lemma \ref{lem:split} below, $U$ is a viscosity subsolution of \eqref{eq0:MA=1} in $D$, and the proof is finished.
\end{proof}

We collect some preliminary lemmas for the viscosity solutions, which will be used in the proof of this section.

\begin{lemma}\label{lem:split}
  Let $D$ and $D_1$ are two domains such that $D\subset\subset D_1$ and $D$ is bounded. Suppose that $u\in C(D_1)$ is a viscosity subsolution of \eqref{eq0:MA=1} in $D_1$, and $v\in C(\overline{D})$ is a viscosity subsolution of \eqref{eq0:MA=1} in $D$. Assume that $v\leq u$ on $\partial D$. If
  \begin{equation*}
  w=
  \begin{cases}
    \max\{u,v\} & \mbox{in } \overline{D}, \\
    u & \mbox{in } D_1\setminus D, \\
  \end{cases}
  \end{equation*}
  is locally convex in $D_1$, then $w$ is a viscosity subsolution of \eqref{eq0:MA=1} in $D_1$.
\end{lemma}
\begin{proof}
  This lemma can be deduced by following the arguments in the proof of \cite[Proposition 2.8]{Caffarelli-Cabre-1995}, where the property is proved for viscosity subsolutions of uniformly elliptic fully nonlinear equations.
\end{proof}

\begin{lemma}\label{lem:max-vissub}
Let $\mathcal{S}$ be a nonempty family of viscosity subsolution of \eqref{eq0:MA=1} in $D$. Define
$$u(x)=\sup\{v(x)|~v\in\mathcal{S}\}\quad\text{for}~x\in D.$$
If $u(x)<\infty$ for $x\in D$. Then $u$ is a viscosity subsolution of \eqref{eq0:MA=1} in $D$.
\end{lemma}
\begin{proof}
Since $v\in\mathcal{S}$ is locally convex in $D$, $u\in C^0(D)$ is locally convex. The rest of the proof follows from that of \text{\cite[Lemma 4.2]{Crandall-Ishii-1992}}.
\end{proof}

\begin{lemma}\label{lem:stablity}
  Assume that $\{u_k\}$ is a family of viscosity solution of \eqref{eq0:MA=1} in $D$ and $u_k\to u$ locally uniformly in $D$ as $k\to\infty$. Then $u$ is a viscosity solution of \eqref{eq0:MA=1} in $D$.
\end{lemma}
\begin{proof}
This property was also mentioned by \cite[Proposition 2.9]{Caffarelli-Cabre-1995} and \cite[Proposition 2.1]{Ishii-1989}, while the proofs therein are omitted. Also, the definition of viscosity solution of \eqref{eq0:MA=1} is a bit different from that of \cite{Caffarelli-Cabre-1995} and \cite{Ishii-1989}.
Hence we give the proof for complements.

Clearly, $u$ is locally convex in $D$. Let $v\in C^2(D)$ and $x_0\in D$ be a local maximum point of $u-v$. Then $x_0$ is a strictly maximum point of $u-v_\varepsilon$, where $v_\varepsilon(x):=v(x)+\varepsilon|x-x_0|^2$, $\varepsilon>0$. Namely, for some $\delta>0$ with $\overline{B_\delta(x_0)}\subset D$,
$$u(x)-v_\varepsilon(x)<u(x_0)-v_\varepsilon(x_0),\quad\forall~0<|x-x_0|\leq\delta.$$
Let $x_k$ be the maximum point of $u_k-v_\varepsilon$ in $\overline{B_\delta(x_0)}\subset D$. We claim that $x_k\to x_0$ as $k\to\infty$. Indeed, we need to prove that for $k$ sufficiently large, $|x_k-x_0|\leq\frac{1}{k}$.
For $x\in \overline{B_\delta(x_0)}\setminus B_{\frac{1}{k}}(x_0)$, the locally uniform convergence of $\{u_k\}$ implies that, for $k$ sufficiently large,
\begin{equation*}
\begin{split}
u_k(x)-v_\varepsilon(x)& <u(x)-v_\varepsilon(x)+\sigma<u(x_0)-v_\varepsilon(x_0)-\sigma\\
& <u_k(x_0)-v_\varepsilon(x_0)\leq u_k(x_k)-v_\varepsilon(x_k),
\end{split}
\end{equation*}
where
$$\sigma=\frac{1}{4}\min_{\overline{B_\delta(x_0)}\setminus B_{\frac{1}{k}}(x_0)}(u(x_0)-v_\varepsilon(x_0)-(u-v_\varepsilon))>0.$$
This gives $x_k\in\overline{B_{\frac{1}{k}}(x_0)}$. Thus,
$$\det(D^2v_\varepsilon(x_k))\geq1.$$
Letting $\varepsilon\to0$ and $k\to\infty$, we have
$$\det(D^2v(x_0))\geq1.$$
Hence, $u$ is a viscosity subsolution of \eqref{eq0:MA=1} in $D$. Similar arguments leads to $u$ being a supersolution.
\end{proof}

With these preliminaries, we are now ready to deal with the Perron's solution. For a bounded open set $\Omega$ of $\mathbb{R}^n$, a function $\varphi$ defined on $\partial\Omega$ and a constant $c$, we begin with the special case of $A=I$ and $b=0$, and denote the set of subfunctions by
\begin{equation*}
\begin{split}
\mathcal{S}_c^{\varphi}=\bigg\{& v\in C^0(\mathbb{R}^n\setminus \Omega)|\ v\text{ is a viscosity subsolution of }\det(D^2v)=1~\text{in }\mathbb{R}^n\setminus\overline{\Omega}, \\
&\ v\leq\varphi\text{ on }\partial\Omega,\ \limsup_{|x|\to\infty}\bigg(v(x)-\frac{1}{2}|x|^2\bigg)\leq c\bigg\}.
\end{split}
\end{equation*}
One basic result is that the Perron's solution satisfies the Monge--Amp\`{e}re equation in the exterior domain in the viscosity sense.
\begin{prop}\label{lem:det=1}
  Let $\Omega$ be a bounded open set of $\mathbb{R}^n$, $n\geq3$, and let $\varphi$ be a bounded function on $\partial\Omega$. Then for every constant $c\geq\sup_{x\in\partial\Omega}\left(\varphi(x)-\frac{1}{2}|x|^2\right)$, the function
  \begin{equation}\label{eq:perron-sol}
  u(x)=\sup\{v(x)|\ v\in\mathcal{S}_c^{\varphi}\}\quad\text{for }x\in\mathbb{R}^n\setminus\overline{\Omega}
  \end{equation}
  is a viscosity solution of \eqref{eq0:MA=1} in $\mathbb{R}^n\setminus\overline{\Omega}$.
\end{prop}

\begin{proof}
The proof will be divided into three steps.

\textbf{Step 1.} We first verify that $u$ is well defined in $\mathbb{R}^n\setminus\overline{\Omega}$. Namely, the supremum is meaningful. It is clear from $c\geq -C_0$ that
$$v^-(x):=\frac{1}{2}|x|^2-C_0\in\mathcal{S}_c^\varphi\quad\text{for }x\in\mathbb{R}^n\setminus\Omega,$$
where $C_0=\sup_{x\in\partial\Omega}\left(\frac{1}{2}|x|^2-\varphi(x)\right)$. Thus, $\mathcal{S}_c^{\varphi}$ is nonempty. Moreover, set
$$v^+(x)=\frac{1}{2}|x|^2+c\quad\text{for }x\in\mathbb{R}^n\setminus\Omega.$$
It is easily seen that $v^+\geq\varphi$ on $\partial\Omega$. It follows from the comparison principle that, for each $v\in\mathcal{S}_c^\varphi$,
$$v\leq v^+\quad\text{in }\mathbb{R}^n\setminus\Omega.$$
Hence, the supremum is well defined, and $u\leq v^+$ in $\mathbb{R}^n\setminus\overline{\Omega}$.

\textbf{Step 2.} We next prove that $u\in C^0(\mathbb{R}^n\setminus\overline{\Omega})$ is a viscosity subsolution of \eqref{eq0:MA=1}. Indeed, this follows from the definition of $u$ and Lemma \ref{lem:max-vissub} immediately.

\textbf{Step 3.} We proceed to prove that $u$ is a viscosity solution of \eqref{eq0:MA=1}. It suffices to prove that, for any $B_r(x_0)\subset\subset\mathbb{R}^n\setminus\overline{\Omega}$,
$$\det{D^2u}=1\quad\text{in }B_r(x_0)$$
in the viscosity sense.

Fix $B_r(x_0)\subset\subset\mathbb{R}^n\setminus\overline{\Omega}$. By the definition of $u$, there exists a sequence $\{v_k\}\subset\mathcal{S}_c^\varphi$ such that $u(x_0)=\lim_{k\to\infty}v_k(x_0)$.
Let $\overline{v}_k=\max\{v_k,v^-\}$. Since $v_k$, $v^-\in \mathcal{S}_c^\varphi$, $\overline{v}_k\in C^0(\mathbb{R}^n\setminus\Omega)$ is a viscosity subsolution of \eqref{eq0:MA=1} in $\mathbb{R}^n\setminus\overline{\Omega}$ satisfying
$$\overline{v}_k\leq\varphi\ \text{on }\partial\Omega\quad\text{and}\quad\limsup_{|x|\to\infty}\bigg(\overline{v}_k(x)-\frac{1}{2}|x|^2\bigg)\leq c.$$
We thus get $\overline{v}_k\in\mathcal{S}_c^\varphi$ and
$$u(x_0)=\lim_{k\to\infty}\overline{v}_k(x_0).$$
Consider the lifting function $\overline{w}_k$ of $\overline{v}_k$ with respect to \eqref{eq0:MA=1} in $B_r(x_0)$. It follows from Proposition \ref{prop:lifting} that $\overline{w}_k\in\mathcal{S}_c^\varphi$ and
$$\overline{v}_k\leq\overline{w}_k\leq u\quad\text{in }B_r(x_0).$$
Hence,
$$u(x_0)=\lim_{k\to\infty}\overline{w}_k(x_0).$$
Note that $v^-\leq\overline{w}_k\leq v^+$ in $B_r(x_0)$. Hence, $\{\overline{w}_k\}$ is uniformly bounded in $B_r(x_0)$. By the local Lipschitz estimate for convex functions (see \cite[Theorem 1 in Chapter 6.3]{Evans-Gariepy1992}), we further have $\{|D\overline{w}_k|\}$ is locally uniformly bounded in $B_r(x_0)$. So, up to a subsequence, $\{\overline{w}_k\}$ locally uniformly converges to some convex function $\overline{w}$ in $B_r(x_0)$. By Lemma \ref{lem:stablity}, $\overline{w}$ is a viscosity solution of \eqref{eq0:MA=1} in $B_r(x_0)$. It follows that
$$\overline{w}\leq u\ \text{in }B_r(x_0)\quad\text{and}\quad\overline{w}(x_0)=u(x_0).$$

It remains to prove $\overline{w}\geq u$ in $B_r(x_0)$. Take $y\in B_r(x_0)$. There exists a sequence $\{V_k\}\subset\mathcal{S}_c^\varphi$ such that $u(y)=\lim_{k\to\infty}V_k(y)$. Let
$$\overline{V}_k=\max\{V_k,\overline{v}_k\}.$$
Then $\overline{V}_k\in\mathcal{S}_c^\varphi$. Consider the lifting function $\overline{W}_k$ of $\overline{V}_k$ with respect to \eqref{eq0:MA=1} in $B_r(x_0)$.
It follows from Proposition \ref{prop:lifting} that $\overline{W}_k\in\mathcal{S}_c^\varphi$ and
$$V_k\leq\overline{V}_k\leq\overline{W}_k\leq u\quad\text{in }B_r(x_0).$$
Hence,
$$u(y)=\lim_{k\to\infty}\overline{W}_k(y).$$
Also note that $v^-\leq\overline{W}_k\leq v^+$ in $B_r(x_0)$. Similar arguments for $\{\overline{w}_k\}$ apply to $\{\overline{W}_k\}$ gives that $\{\overline{W}_k\}$ locally uniformly converges to some convex
function $\overline{W}$ in $B_r(x_0)$, and $\overline{W}$ is a viscosity solution of \eqref{eq0:MA=1} in $B_r(x_0)$ satisfying
$$\overline{W}(y)=u(y).$$
It follows that
$$\overline{v}_k\leq\overline{V}_k,\quad\overline{w}_k\leq\overline{W}_k,\quad\overline{w}\leq\overline{W}\quad\text{in } B_r(x_0).$$
Then
$$u(x_0)=\overline{w}(x_0)\leq\overline{W}(x_0)\leq u(x_0).$$

We conclude that
\begin{equation*}
\begin{cases}
\det(D^2\overline{w})=\det(D^2\overline{W})=1& \text{in }B_r(x_0), \\
\overline{w}\leq\overline{W}& \text{in }B_r(x_0), \\
\overline{w}(x_0)=\overline{W}(x_0).
\end{cases}
\end{equation*}
By the interior regularity of viscosity solution of \eqref{eq0:MA=1} (see for instance \cite{Caffarelli1995}), one has $\overline{w}$, $\overline{W}\in C^\infty(\overline{B_{\frac{r}{2}}(x_0)})$. This yields that
$$\sum_{i,j=1}^na_{ij}(x)D_{ij}(\overline{W}-\overline{w})(x)=0\quad\text{for }x\in B_{\frac{r}{2}}(x_0),$$
where
$$a_{ij}(x)=\int_{0}^{1}\frac{\partial (\det)^{\frac{1}{n}}}{\partial\xi_{ij}}(D^2\overline{w}+t(D^2\overline{W}-D^2\overline{w}))\mathrm{d}t.$$
Since $D^2\overline{w}$ and $D^2\overline{W}$ are bounded in $\overline{B_{\frac{r}{2}}(x_0)}$, we have $(a_{ij}(x))\geq\delta I$ in $B_{\frac{r}{2}}(x_0)$ for some constant $\delta>0$.
The strong maximum principle for uniformly elliptic linear equation implies that
$$\overline{W}=\overline{w}\quad\text{in }B_r(x_0).$$
This gives $\overline{w}(y)=\overline{W}(y)=u(y)$.
By the arbitrariness of $y$, we have
$$u=\overline{w}\quad\text{in }B_r(x_0).$$
This completes the proof.
\end{proof}

\begin{remark}
We see from the proof that if $\varphi(x)-\frac{1}{2}|x|^2$ is bounded on $\partial\Omega$, then Proposition \ref{lem:det=1} also holds for unbounded $\Omega$.

\end{remark}

\section{Asymptotic behavior of the Perron's solution near infinity}\label{sec:asy}

In this section, we will demonstrate that the Perron's solution achieves the asymptotic behavior $\frac{1}{2}|x|^2+c$ near infinity, provided $c$ large.

\begin{prop}\label{lem:asym}
  Let $\Omega$ be a bounded open set of $\mathbb{R}^n$ satisfying $B_{r_1}(0)\subset\Omega\subset\subset B_{r_2}(0)$ for some positive constants $r_1$ and $r_2$, $n\geq3$, and let $\varphi$ be a bounded function on $\partial\Omega$. Then the function $u$ defined by \eqref{eq:perron-sol} satisfies
  $$\lim_{|x|\to\infty}\left(u(x)-\left(\frac{1}{2}|x|^2+c\right)\right)=0,$$
  where
  \begin{equation}\label{eq:c1}
  c\geq\max\bigg\{\mu(-r_1^n), \sup_{x\in\partial\Omega}\bigg(\varphi(x)-\frac{1}{2}|x|^2\bigg)\bigg\},
  \end{equation}
  and
  $$\mu(\alpha)=\inf_{\partial\Omega}\varphi-\frac{1}{2}r_2^2+\int_{r_2}^\infty s\bigg(\bigg(1+\frac{\alpha}{s^n}\bigg)^{\frac{1}{n}}-1\bigg)\mathrm{d}s\quad \text{for}~\alpha\geq-r_1^n.$$
\end{prop}
\begin{proof}
For $\alpha\geq-r_1^n$, let
\begin{equation}\label{eq:w-alpha}
w_{\alpha}(x)=\inf_{\partial\Omega}\varphi+\int_{r_2}^{|x|}(s^n+\alpha)^{\frac{1}{n}}\mathrm{d}s\quad\text{for }x\in\mathbb{R}^n\setminus\Omega.
\end{equation}
Then $w_{\alpha}\in C^0(\mathbb{R}^n\setminus\Omega)$ is locally convex in $\mathbb{R}^n\setminus\overline{\Omega}$, and
\begin{equation*}
  \begin{cases}
    \det{D^2w_{\alpha}}=1 & \mbox{in }\mathbb{R}^n\setminus\overline{\Omega}, \\
    w_{\alpha}\leq\varphi & \mbox{on }\partial\Omega.
  \end{cases}
\end{equation*}
Direct calculation shows
\begin{equation*}
\begin{split}
&\lim_{|x|\to\infty}\bigg(w_{\alpha}(x)-\frac{1}{2}|x|^2\bigg) \\
= & \inf_{\partial\Omega}\varphi-\frac{1}{2}r_2^2+\int_{r_2}^\infty s\bigg(\bigg(1+\frac{\alpha}{s^n}\bigg)^{\frac{1}{n}}-1\bigg)\mathrm{d}s \\
= & :\mu(\alpha).
\end{split}
\end{equation*}
Clearly, $\mu(\alpha)$ is smooth and strictly increasing with respect to $\alpha\in[-r_1^n,\infty)$ and
$$\lim_{\alpha\to\infty}\mu(\alpha)=\infty.$$
Then for $c\geq\mu(-r_1^n)$, we have $w_{\mu^{-1}(c)}\in\mathcal{S}_c^{\varphi}$.
It follows that
$$\liminf_{|x|\to\infty}\bigg(u(x)-\frac{1}{2}|x|^2\bigg)\geq\lim_{|x|\to\infty}\bigg(w_{\mu^{-1}(c)}(x)-\frac{1}{2}|x|^2\bigg)=c.$$
Recall that for $c\geq\sup_{x\in\partial\Omega}\left(\varphi(x)-\frac{1}{2}|x|^2\right)$,
$$\limsup_{|x|\to\infty}\bigg(u(x)-\frac{1}{2}|x|^2\bigg)\leq\lim_{|x|\to\infty}\bigg(v^+(x)-\frac{1}{2}|x|^2\bigg)=c,$$
where $v^+(x)=\frac{1}{2}|x|^2+c$ is as in the proof of Proposition \ref{lem:det=1}.
This finishes the proof.
\end{proof}
\section{Boundary behavior of the Perron's solution}\label{sec:bdy}

In this section, we deal with the boundary behavior of the Perron's solution. The key lies in proving a barrier lemma. We shall overcome the difficulty of non $C^2$ regularity of boundary values and domains.

Let $\varphi$ be semi-convex with respect to $\partial\Omega$ at $\xi$. Namely, $\psi(x'):=\varphi(x',\rho(x'))$ is semi-convex in $B'_{\delta(\xi)}(0')$ under the local coordinate system at $\xi\in\partial\Omega$, where $\delta(\xi)$ is as in \eqref{eq:bdyre}. Then there exist $K(\xi)>0$ and $p(\xi)\in\mathbb{R}^{n-1}$ such that
  \begin{equation}\label{eq:semi-convex}
  \psi(x')+\frac{K(\xi)}{2}|x'|^2\geq\psi(0')+p(\xi)\cdot x'\quad\text{for }|x'|<\delta(\xi).
  \end{equation}
By the semi-convexity of $\psi$ at $\xi$, $\psi$ is continuous in $B'_{\delta(\xi)}(0')$, and so $\varphi$ is continuous on $\partial\Omega\cap B_{\delta(\xi)}(0)$.

Now, we are able to prove the following barrier lemma.

\begin{lemma}\label{lem:barrier}
Let $\Omega$ be a convex domain, and let $\varphi$ be a bounded function on $\partial\Omega$.
Suppose that $\Omega$ satisfies an enclosing sphere condition at $\xi\in\partial\Omega$, and
$\varphi$ is semi-convex with respect to $\partial\Omega$ at $\xi$. Then there exists $\bar{x}(\xi)\in\mathbb{R}^n$ such that
$$|\bar{x}(\xi)|\leq C(\xi)\quad\text{and}\quad w_{\xi}<\varphi\quad\text{on }\partial\Omega\setminus\{\xi\},$$
where
$$w_{\xi}(x)=\varphi(\xi)+\frac{1}{2}\left(|x-\bar{x}(\xi)|^2-|\xi-\bar{x}(\xi)|^2\right)\quad\text{for}\ x\in\mathbb{R}^n,$$
and $C(\xi)>0$ is a constant depending only on $\delta(\xi)$, $r(\xi)$, $K(\xi)$, $|p(\xi)|$, $\sup_{\partial\Omega}|\varphi|$ and the $C^{0,1}$ norm of $\partial\Omega$. Here $\delta(\xi)$ and $r(\xi)$ are as in \eqref{eq:bdyre} and Definition \ref{defn:exterior-shpere} respectively, $K(\xi)$ and $p(\xi)$ are as in \eqref{eq:semi-convex}.
\end{lemma}

\begin{proof}
  By a translation and a rotation, we may assume without losing the generality that the coordinate system in $\mathbb{R}^n$ is just the local coordinate system at $\xi$, and $x_n$-axis is along the direction of $y(\xi)-\xi$, where $y(\xi)$ is the center of the enclosing sphere at $\xi$. Then $\xi=0$, $y(\xi)=(0',r(\xi))$ and $\partial\Omega$ can be locally represented by the graph of $x_n=\rho(x')$, $|x'|<\delta(\xi)$.
  Let
  $$\bar{x}=(-p(\xi),R),$$
  where $R>0$ will be chosen later. Let
  $$w(x)=\varphi(0)+\frac{1}{2}\left(|x-\bar{x}|^2-|\bar{x}|^2\right)\quad\text{for}\ x\in\mathbb{R}^n.$$
  Denote $\delta=\delta(\xi)$, $r=r(\xi)$, $K=K(\xi)$ and $p=p(\xi)$. It is clear from the enclosing sphere condition $\partial\Omega\subset\overline{B_r((0',r))}$ that
  $$|x'|^2+(\rho(x')-r)^2\leq r^2\quad\text{for}~|x'|<\delta,$$
  and so
  \begin{equation}\label{eq:exterior-sphere-convex}
  \rho(x')\geq\frac{1}{2r}|x'|^2\quad\text{for}~|x'|<\delta.
  \end{equation}
  Since $\Omega$ is bounded and convex, $\partial\Omega$ is Lipschitz continuous and
  \begin{equation}\label{eq:bdy-Lip}
    \rho(x')\leq C_0|x'|\quad\text{for }|x'|<\delta,
  \end{equation}
  where $C_0>0$ is the $C^{0,1}$ norm of $\partial\Omega$.

  \textbf{Case 1.} $x=(x',\rho(x'))\in\partial\Omega\setminus\{0\}$ and $|x'|<\delta$. It follows from \eqref{eq:semi-convex}-\eqref{eq:bdy-Lip} that
  \begin{equation*}
    \begin{split}
       (w-\varphi)(x)= &\ \frac{1}{2}|x|^2-x\cdot\bar{x}+\varphi(0)-\varphi(x) \\
        = &\ \frac{1}{2}|x'|^2+\frac{1}{2}\rho(x')^2-R\rho(x')+\varphi(0)-\varphi(x',\rho(x'))+p\cdot x' \\
        \leq &\ \frac{1}{2}\bigg(1+C_0^2-\frac{R}{r}\bigg)|x'|^2+\psi(0')-\psi(x')+p\cdot x' \\
        \leq &\ \frac{1}{2}\bigg(1+C_0^2-\frac{R}{r}+K\bigg)|x'|^2 \\
        < &\ 0,
    \end{split}
  \end{equation*}
  provided $$R>R_1:=r(1+C_0^2+K).$$

  \textbf{Case 2.} $x\in\partial\Omega\setminus\{(x',
  \rho(x'))|\,|x'|<\delta\}$. By the convexity of $\Omega$ and \eqref{eq:exterior-sphere-convex}, we have $x_n\geq\frac{1}{2r}\delta^2$.
  It follows that
  \begin{equation*}
    \begin{split}
       (w-\varphi)(x) = &\ \frac{1}{2}|x|^2-x\cdot\bar{x}+\varphi(0)-\varphi(x) \\
       = &\ \frac{1}{2}|x|^2+p\cdot x'-Rx_n+\varphi(0)-\varphi(x) \\
        \leq &\ \frac{1}{2}(2r)^2+r|p|-\frac{R}{2r}\delta^2+2\sup_{\partial\Omega}|\varphi| \\
        < &\ 0,
    \end{split}
  \end{equation*}
  provided $$R>R_2:=\frac{2r}{\delta^2}(2r^2+r|p|+2\sup_{\partial\Omega}|\varphi|).$$

  Take $R>\max\{R_1,R_2\}$. In both cases, we conclude $w<\varphi$ on $\partial\Omega\setminus\{\xi\}$. We finish the proof by taking $C(\xi)=\sqrt{|p|^2+R^2}$.
\end{proof}


Benefiting from Lemma \ref{lem:barrier}, we obtain that the Perron's solution can be continuously extended to the boundary $\partial\Omega$ in a pointwise way.

\begin{prop}\label{lem:perron0}
Let $u$ be the function defined by \eqref{eq:perron-sol}. Suppose that $\Omega$, $\varphi$ and $\xi\in\partial\Omega$ satisfy the assumptions in Lemma \ref{lem:barrier}. Then there exists a constant $c_0(\xi)$, depending only on $n$, $\delta(\xi)$, $r(\xi)$, $K(\xi)$, $|p(\xi)|$, $\sup_{\partial\Omega}|\varphi|$ and the $C^{0,1}$ norm of $\partial\Omega$, such that for every $c>c_0(\xi)$, $u(x)\to\varphi(\xi)$ as $\mathbb{R}^n\setminus\overline{\Omega}\ni x\to\xi$. 
\end{prop}

\begin{proof}
Since $\Omega$ satisfies an enclosing sphere condition at $\xi$, $\Omega$ is bounded. By a translation, we may assume $B_{r_1}(0)\subset\Omega\subset\subset B_{r_2}(0)$ for some constants $r_1,r_2>0$. By Lemma \ref{lem:barrier}, we have
\begin{equation*}
\begin{split}
w_{\xi}(x) & =\varphi(\xi)+\frac{1}{2}|x|^2-\frac{1}{2}|\xi|^2+(\xi-x)\cdot\bar{x}(\xi) \\
& \geq \varphi(\xi)-\frac{1}{2}r_2^2-2r_2C(\xi) \\
& \geq \varphi(\xi)-C_1\quad\text{for}~x\in B_{r_2}(0)\setminus\Omega,
\end{split}
\end{equation*}
and
\begin{equation*}
\begin{split}
w_\xi(x) & =\varphi(\xi)+\frac{1}{2}|x|^2-\frac{1}{2}|\xi|^2+(\xi-x)\cdot\bar{x}(\xi) \\
& \leq\varphi(\xi)+\frac{1}{2}(r_2+1)^2+(2r_2+1)C(\xi) \\
& \leq \varphi(\xi)+C_1\quad\text{for}~x\in B_{r_2+1}(0)\setminus B_{r_2}(0),
\end{split}
\end{equation*}
where $C(\xi)$ is as in Lemma \ref{lem:barrier} and $C_1$ depends only on $C(\xi)$ and $r_2$. Take $\alpha_0\geq-r_1^n$ such that
$$\inf_{\partial\Omega}\varphi+\int_{r_2}^{r_2+1}\left(s^n+\alpha_0\right)^{\frac{1}{n}}\mathrm{d}s>\sup_{\partial\Omega}\varphi+2C_1,$$
and let
$$\underline{w}_{\alpha}(x)=w_{\alpha}(x)-C_1,$$
where $w_\alpha$ is given by \eqref{eq:w-alpha}. Then for any $\alpha\geq\alpha_0$,
$$w_{\xi}\geq\inf_{\partial\Omega}\varphi-C_1\geq\underline{w}_{\alpha}\quad\text{in}~B_{r_2}(0)\setminus\Omega,$$
$$w_{\xi}<\underline{w}_{\alpha_0}\leq\underline{w}_{\alpha}\quad\text{on}~\partial B_{r_2+1}(0).$$

Let
\begin{equation*}
\underline{u}(x)=
\begin{cases}
w_{\xi}(x), & x\in B_{r_2}(0)\setminus\Omega, \\
\max\{\underline{w}_{\alpha}(x),w_{\xi}(x)\}, & x\in B_{r_2+1}(0)\setminus B_{r_2}(0), \\
\underline{w}_{\alpha}(x), & x\in\mathbb{R}^n\setminus B_{r_2+1}(0).
\end{cases}
\end{equation*}
In view of Lemma \ref{lem:split}, $\underline{u}\in\mathcal{S}_c^\varphi$ provided $c\geq\mu(\alpha_0)$. Hence,
$$\liminf_{\mathbb{R}^n\setminus\overline{\Omega}\ni x\to\xi}u(x)\geq\liminf_{\mathbb{R}^n\setminus\overline{\Omega}\ni x\to\xi}\underline{u}(x)=w_{\xi}(\xi)=\varphi(\xi).$$

Recall that $\varphi$ is continuous in a neighborhood of $\xi$ on $\partial\Omega$. Since $\varphi$ is bounded, the extension theorem implies that there exists $\overline{\varphi}\in C^0(\partial\Omega)$ satisfying $\overline{\varphi}=\varphi$ near $\xi$ and
$$\overline{\varphi}\geq\varphi\quad \text{on }\partial\Omega.$$
Since $\Omega$ is bounded and convex, $\partial\Omega$ is Lipschitz continuous. Then $B_{r_2}(0)\setminus\overline{\Omega}$ satisfies the exterior cone condition. By Proposition \hyperlink{B}{B}, there exists $w^+\in C^0(\overline{B_{r_2}(0)\setminus\Omega})$ satisfying
\begin{equation*}
  \begin{cases}
    \Delta w^+=0 & \mbox{in } B_{r_2}(0)\setminus\overline{\Omega}, \\
    w^+=\overline{\varphi} & \mbox{on }\partial\Omega, \\
    w^+=\max_{\partial B_{r_2}(0)}u & \mbox{on }\partial B_{r_2}(0).
  \end{cases}
\end{equation*}
Then comparison principal gives that for any $v\in\mathcal{S}_c^\varphi$,
$$v\leq w^+\quad\text{in }B_{r_2}(0)\setminus\overline{\Omega}.$$
It follows that
$$u\leq w^+\quad\text{in }B_{r_2}(0)\setminus\overline{\Omega},$$
and so
$$\limsup_{\mathbb{R}^n\setminus\overline{\Omega}\ni x\to\xi}u(x)\leq\lim_{\mathbb{R}^n\setminus\overline{\Omega}\ni x\to\xi}w^+(x)=\overline{\varphi}(\xi)=\varphi(\xi).$$
Therefore, $u$ satisfies the boundary condition at $\xi$.
\end{proof}

\section{Proof of Theorem \ref{thm:main}}\label{sec:main}

In the previous section, we proved that the Perron's solution is continuous up to a boundary point $\xi$ when $c>c_0(\xi)$. In this section, we will give the uniform estimates for $c_0(\xi)$ with respect to $\xi\in\partial\Omega$, and thus obtain that the Perron's solution is continuous up to the whole boundary when $c$ is sufficiently large. Combining with the conclusions in Sections \ref{sec:eq}-\ref{sec:bdy}, we can complete the proof of the existence part of Theorem \ref{thm:main}. The nonexistence part can be deduced as in \cite{Li-Lu}.

To establish the uniform estimate for $c_0(\xi)$, we need the following uniform estimates for $\delta(\xi)$, $K(\xi)$ and $p(\xi)$, where $\delta(\xi)$ is as in \eqref{eq:bdyre}, and $K(\xi)$ and $p(\xi)$ are as in \eqref{eq:semi-convex}. These estimates are also of independent interest.

\begin{lemma}\label{lem:A}
 Let $\Omega$ be a bounded open set of $\mathbb{R}^n$, $\partial\Omega\in C^{1}$. Then $\Omega$ satisfies
 \begin{equation}\label{H}
 \delta:=\inf_{\partial\Omega}\delta(\xi)>0.\tag{H}
 \end{equation}
\end{lemma}
\begin{proof}
  By the finite covering theorem, there exists $\{\xi^{(1)},\cdots,\xi^{(N)}\}\subset\partial\Omega$ such that $$\partial\Omega=\bigcup_{i=1}^N\bigg\{((x^{(i)})',\rho^{(i)}((x^{(i)})'))|\,|(x^{(i)})'|<\frac{\delta_i}{4}\bigg\},$$
  where $x^{(i)}$ is the coordinate under the local coordinate system at $\xi^{(i)}$ and $\delta_i=\delta(\xi^{(i)})$.
  We may assume $\omega(\delta_i)<\frac{1}{2}$, where $\omega$ is the modulus of the continuity of $D'\rho$.
  Let $\delta=\frac{1}{8}\min\{\delta_1,\cdots,\delta_N\}$.

  For any fixed $\widetilde{\xi}\in\partial\Omega$, there exists $i_0\in\{1,\cdots,N\}$ such that
  $$\widetilde{\xi}\in\bigg\{(x',\rho(x'))|\,|x'|<\frac{\delta_{i_0}}{4}\bigg\}.$$
  Here we write the local coordinate $((x^{(i_0)})',\rho^{(i_0)}((x^{(i_0)})'))$ as $(x',\rho(x'))$ for simplicity. We may assume the coordinate system in $\mathbb{R}^n$ is just the local coordinate system at $\xi^{(i_0)}$. Since $\partial\Omega$ is $C^{1}$, we have
  \begin{equation}\label{eq:Hsemi}
  |D'\rho(\widetilde{\xi}')|=|D'\rho(\widetilde{\xi}')-D'\rho(0')|\leq \omega(|\widetilde{\xi}'|)\leq\omega\bigg(\frac{\delta_{i_0}}{4}\bigg).
  \end{equation}
  Clearly, the unit inner normal of $\partial\Omega$ at $0$ and $\widetilde{\xi}$ are
  $$\nu(0)=(0',1)\quad\text{and}\quad\nu(\widetilde{\xi})=\frac{(-D'\rho(\widetilde{\xi'}),1)}{\sqrt{|D'\rho(\widetilde{\xi}')|^2+1}},$$
  respectively. It follows from \eqref{eq:Hsemi} that
  \begin{equation}\label{eq:nu2}
  |\nu(\widetilde{\xi})-\nu(0)|=\sqrt{2\Bigg(1-\frac{1}{\sqrt{|D'\rho(\widetilde{\xi}')|^2+1}}\Bigg)}
  <|D'\rho(\widetilde{\xi}')|\leq\omega\bigg(\frac{\delta_{i_0}}{4}\bigg).
  \end{equation}

\begin{figure}[h]
  \centering
  \includegraphics[width=7.5cm]{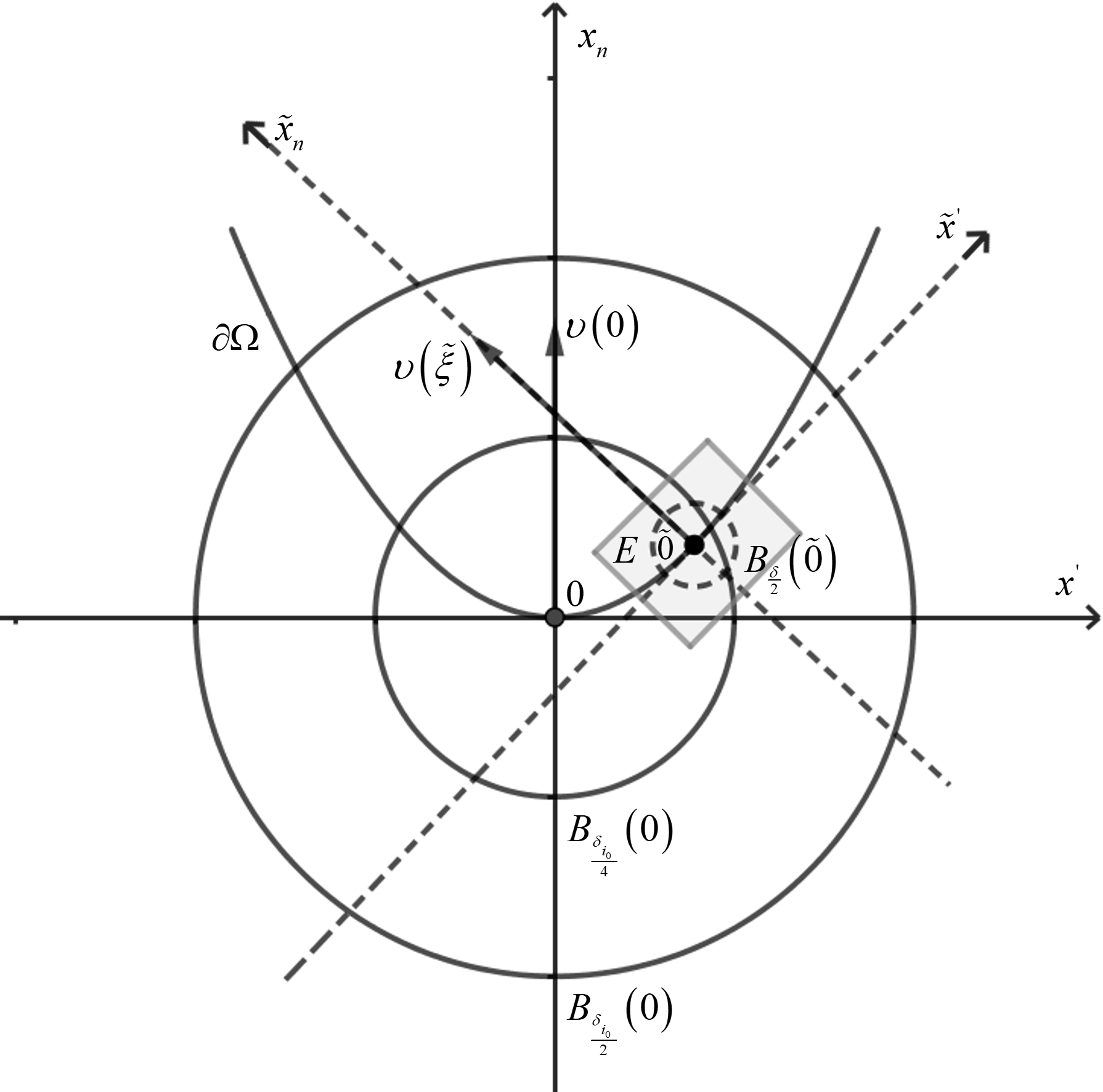}
\end{figure}

  Denote by $\widetilde{x}$ the coordinate under the local coordinate system at $\widetilde{\xi}=\widetilde{0}$. Then there exists a $n\times n$ orthogonal matrix $Q=(q_{ij})$ depending only on $\widetilde{\xi}$ such that
  \begin{equation}\label{eq:trans}
  \widetilde{x}=(x-\widetilde{\xi})Q.
  \end{equation}
  Note that the coordinate of $\nu(\widetilde{\xi})$ under the local coordinate system at $\widetilde{\xi}$ is $(0',1)$. This yields that
  $$\widetilde{\nu}(\widetilde{\xi})=\nu(\widetilde{\xi})Q,$$
  and so
  $$\nu(\widetilde{\xi})=\nu(0)Q^{-1}=\nu(0)Q^T=(0',1)Q^T=(q_{1n}, \cdots,q_{(n-1)n},q_{nn}).$$
  Combining with \eqref{eq:nu2}, we get
  $$|(q_{1n},\cdots,q_{(n-1)n})|<w\bigg(\frac{\delta_{i_0}}{4}\bigg)\quad\text{and}\quad1-w\bigg(\frac{\delta_{i_0}}{4}\bigg)<q_{nn}\leq1.$$
  It follows that
  \begin{equation}\label{eq:ithm}
  (\widetilde{x}Q^{-1})'=(\widetilde{x}Q^{T})'=\widetilde{x}'(Q^{T})'+\widetilde{x}_n(q_{1n},\cdots,q_{(n-1)n}),
  \end{equation}
  where $(Q^{T})'$ denotes the matrix composed of the first $(n-1)$ rows and columns of $Q^T$. We see for $(\widetilde{x}',\widetilde{x}_n)\in E:=\Big\{(\widetilde{x}',\widetilde{x}_n)|\,|\widetilde{x}'|<\delta,~|\widetilde{x}_n|<\frac{\delta}{w\big(\frac{\delta_{i_0}}{4}\big)}\Big\}$,
  \begin{equation*}
  \begin{split}
  |x'|=|(\widetilde{x}Q^{-1})'+\widetilde{\xi}'| & \leq|\widetilde{x}'(Q^{T})'|+|\widetilde{x}_n(q_{1n},\cdots,q_{(n-1)n})|+|\widetilde{\xi}'| \\
  & \leq |\widetilde{x}'|+w\bigg(\frac{\delta_{i_0}}{4}\bigg)|\widetilde{x}_n|+|\widetilde{\xi}'| \\
  & <\frac{\delta_{i_0}}{2}.
  \end{split}
  \end{equation*}

  For $(\widetilde{x}',\widetilde{x}_n)\in E$, set
  $$F(\widetilde{x}',\widetilde{x}_n)=\rho((\widetilde{x}Q^{-1}+\widetilde{\xi})')-(\widetilde{x}Q^{-1}+\widetilde{\xi})_n.$$
  From $\rho\in C^1$, we have $F\in C^1(E)$. Since for $(\widetilde{x}',\widetilde{x}_n)\in E$, $|x'|<\frac{\delta_{i_0}}{2}$, and so $|D'\rho(x')|\leq w\Big(\frac{\delta_{i_0}}{2}\Big)$. It follows that
  \begin{equation*}
  \begin{split}
  \frac{\partial F}{\partial\widetilde{x}_n} & =D'\rho((\widetilde{x}Q^{-1}+\widetilde{\xi})')\cdot(q_{1n},\cdots,q_{(n-1)n})-q_{nn} \\
  & \leq w\bigg(\frac{\delta_{i_0}}{2}\bigg)w\bigg(\frac{\delta_{i_0}}{4}\bigg)-\bigg(1-w\bigg(\frac{\delta_{i_0}}{4}\bigg)\bigg) \\
  & <0,
  \end{split}
  \end{equation*}

  \noindent due to $w(\frac{\delta_{i_0}}{2})<\frac{1}{2}$.
  In view of \eqref{eq:trans}, we have
  $$(\widetilde{x}Q^{-1}+\widetilde{\xi})_n=x_n=\rho(x')=\rho((\widetilde{x}Q^{-1}+\widetilde{\xi})').$$
  That is $F=0$ in $E\cap\{(\widetilde{x}',\widetilde{x}_n)|\,\widetilde{x}=(x-\widetilde{\xi})Q\}$. Therefore, the implicit function theorem yields that there exists a function $\widetilde{\rho}$ such that
  $$\widetilde{x}_n=\widetilde{\rho}(\widetilde{x}')\quad\text{for}~|\widetilde{x}'|<\frac{\delta}{2}.$$
  This finishes the proof.
\end{proof}

\begin{lemma}\label{prop:convex-C1}
  Let $\Omega$ be a bounded open set of $\mathbb{R}^n$ and $\partial\Omega\in C^1$.
  If $\varphi$ is semi-convex with respect to $\partial\Omega$, then $K$ and $|p|$ are bounded on $\partial\Omega$.
\end{lemma}

\begin{proof}
  For any fixed $\xi\in\partial\Omega$, $\partial\Omega$ can be locally represented by the graph of
  $$x_n=\rho(x')\quad\text{for}~|x'|<\delta.$$
  Here $\delta$ is independent of $\xi$ due to Lemma \ref{lem:A}. Take $\widetilde{\xi}\in\partial\Omega\cap\{(x',\rho(x'))|\,|x'|<\frac{\delta}{2}\}$. Denote by $\widetilde{x}$ the coordinate under the local coordinate system at $\widetilde{\xi}$. Then there exists a $n\times n$ orthogonal matrix $Q$ depending only on the local coordinate systems at $\xi$ and $\widetilde{\xi}$ such that
  \begin{equation*}\label{eq:coordinate-t}
  \widetilde{x}=(x-\widetilde{\xi})Q.
  \end{equation*}
  It is easily seen that for $\widetilde{x}\in\partial\widetilde{\Omega}$ with $|\widetilde{x}'|<\widetilde{\delta}$,
  \begin{equation}\label{eq:x'-x}
  |x'|=|(\widetilde{x}Q^{-1}+\widetilde{\xi})'|\leq|\widetilde{x}|+|\widetilde{\xi}'|<\frac{\delta}{2}+\frac{\delta}{2}=\delta\quad\text{for}\ |\widetilde{x}'|<\widetilde{\delta},
  \end{equation}
  where $\widetilde{\delta}:=\frac{\delta}{2\sqrt{1+C_0^2}}$ and $C_0$ is the $C^{0,1}$ norm of $\partial\Omega$. Correspondingly, we write
  \begin{equation}\label{eq:psi2}
  \widetilde{\psi}(\widetilde{x}')=\psi(x')
  =\psi((\widetilde{x}Q^{-1}+\widetilde{\xi})')\quad\text{for}~|\widetilde{x}'|<\widetilde{\delta}.
  \end{equation}
  By \eqref{eq:semi-convex}, there exist positive constants $K(\xi)$ and $K(\widetilde{\xi})$ such that
  $$\psi(x')+\frac{K(\xi)}{2}|x'|^2\quad\text{and}\quad\widetilde{\psi}(\widetilde{x}')+\frac{K(\widetilde{\xi})}{2}|\widetilde{x}'|^2$$
  are convex in $B'_\delta(0')$ and $B'_{\widetilde{\delta}}(\widetilde{0}')$, respectively.

  We first prove that $K$ is bounded on $\partial\Omega\cap\{(x',\rho(x'))|\,|x'|<\frac{\delta}{2}\}$. For $\widetilde{x},\widetilde{y}\in\partial\Omega$ with $|\widetilde{x}'|,|\widetilde{y}'|<\widetilde{\delta}$ and $0<t<1$, we have by \eqref{eq:x'-x},
  \begin{equation}\label{eq:delta'}
  |x'|=|(\widetilde{x}Q^{-1}+\widetilde{\xi})'|<\delta\quad\text{and}\quad|y'|=|(\widetilde{y}Q^{-1}+\widetilde{\xi})'|<\delta.
  \end{equation}
  It follows from \eqref{eq:psi2}, \eqref{eq:delta'} and the semi-convexity of $\psi$ in $B'_{\delta}(0')$ that

\begin{equation*}
\begin{split}
     \widetilde{\psi}({t\widetilde{x}'+(1-t)\widetilde{y}'})=&\,\psi(t(\widetilde{x}Q^{-1}+\widetilde{\xi})'+(1-t)(\widetilde{y}Q^{-1}+\widetilde{\xi})') \\ &\,+\frac{K(\xi)}{2}|t(\widetilde{x}Q^{-1}+\widetilde{\xi})'+(1-t)(\widetilde{y}Q^{-1}+\widetilde{\xi})'|^2 \\
     &\,-\frac{K(\xi)}{2}|t(\widetilde{x}Q^{-1}+\widetilde{\xi})'+(1-t)(\widetilde{y}Q^{-1}+\widetilde{\xi})'|^2\\
\end{split}
\end{equation*}
\begin{equation*}
\begin{split}
     \leq &\, t\bigg(\psi((\widetilde{x}Q^{-1}+\widetilde{\xi})')+\frac{K(\xi)}{2}|(\widetilde{x}Q^{-1}+\widetilde{\xi})'|^2\bigg) \\ &\,+(1-t)\bigg(\psi((\widetilde{y}Q^{-1}+\widetilde{\xi})')+\frac{K(\xi)}{2}|(\widetilde{y}Q^{-1}+\widetilde{\xi})'|^2\bigg) \\
     &\,-\frac{K(\xi)}{2}|t(\widetilde{x}Q^{-1}+\widetilde{\xi})'+(1-t)(\widetilde{y}Q^{-1}+\widetilde{\xi})'|^2\\
     = &\, t\widetilde{\psi}(\widetilde{x}')+(1-t)\widetilde{\psi}(\widetilde{y}')
     +\frac{t(1-t)K(\xi)}{2}|((\widetilde{x}-\widetilde{y})Q^{-1})'|^2. \\
\end{split}
\end{equation*}
Taking $K(\widetilde{\xi})=(1+C_0^2)K(\xi)$, we obtain for $\widetilde{x}',\widetilde{y}'\in B'_{\widetilde{\delta}}(\widetilde{0})$,
\begin{equation*}
\begin{split}
& \widetilde{\psi}({t\widetilde{x}'+(1-t)\widetilde{y}'})+\frac{K(\widetilde{\xi})}{2}|{t\widetilde{x}'+(1-t)\widetilde{y}'}|^2 \\
\leq &\, t\widetilde{\psi}(\widetilde{x}')+(1-t)\widetilde{\psi}(\widetilde{y}')+\frac{t(1-t)K(\widetilde{\xi})}{2}|\widetilde{x}'-\widetilde{y}'|^2
+\frac{K(\widetilde{\xi})}{2}|{t\widetilde{x}'+(1-t)\widetilde{y}'}|^2 \\
= & \, t\bigg(\widetilde{\psi}(\widetilde{x}')+\frac{K(\widetilde{\xi})}{2}|\widetilde{x}'|^2\bigg)+(1-t)\bigg(\widetilde{\psi}(\widetilde{y}')
+\frac{K(\widetilde{\xi})}{2}|\widetilde{y}'|^2\bigg).
\end{split}
\end{equation*}
Hence, $K$ is uniformly bounded on $\partial\Omega\cap\{(x',\rho(x'))|\,|x'|<\frac{\delta}{2}\}$ and thus is bounded on $\partial\Omega$ by a finite cover argument.

We proceed to prove that $p$ is bounded on $\partial\Omega$.  Since $\psi(x')+\frac{K(\xi)}{2}|x'|^2$ is convex in $B'_{\delta}(0')$, we have for $x'=\frac{\delta p(\xi)}{2|p(\xi)|}\in B'_{\delta}(0')$,
$$|p(\xi)|\leq\frac{2}{\delta}\bigg(\psi(x')-\psi(0')+\frac{K(\xi)}{2}|x'|^2\bigg)\leq\frac{2}{\delta}\bigg(2\sup_{\partial\Omega}|\varphi|+\frac{\delta^2K(\xi)}{8}
\bigg)\leq C,$$
where $C$ is a constant independent of $\xi$. This completes the proof.
\end{proof}

Summing up, we now have all ingredients to present the proof of Theorem \ref{thm:main}. We start the proof by proving a special and simple case of Theorem \ref{thm:main} where $A=I$ and $b=0$.

\begin{prop}\label{lem:main}
  Let $\Omega$ be a domain of $\mathbb{R}^n$ satisfying a uniform enclosing sphere condition,
  $n\geq3$, $\partial\Omega\in C^1$. Let $\varphi$ be semi-convex with respect to $\partial\Omega$. Then there exists some constant $c_*$, such that \begin{equation}\label{eq:DiriProb2}
  \begin{cases}
  \det(D^2u)=1\quad\text{in}~\mathbb{R}^n\setminus\overline{\Omega},\\
  u=\varphi\quad\text{on}~\partial\Omega,\\
  \lim_{|x|\to\infty}\Big(u(x)-\Big(\frac{1}{2}|x|^2+c\Big)\Big)=0
  \end{cases}
  \end{equation}
  has a viscosity solution in $C^0(\mathbb{R}^n\setminus\Omega)$ if and only if $c\geq c_*$, where $c_*$ depends only on $n$, $\Omega$ and $\varphi$.
\end{prop}

\begin{proof}
We divide the proof into four steps.

\textbf{Step 1.} There is a constant $c_1$ such that for $c\geq c_1$, \eqref{eq:DiriProb} has a viscosity solution.
Let
\begin{equation}\label{eq:perron-sol2}
u(x)=\sup\{v(x)|\ v\in\mathcal{S}_c^\varphi\}\quad\text{for }x\in\mathbb{R}^n\setminus\overline{\Omega}.
\end{equation}
By Propositions \ref{lem:det=1} and \ref{lem:asym}, $u$ is a viscosity solution of \eqref{eq0:MA=1} in $\mathbb{R}^n\setminus\overline{D}$ and approaches $\frac{1}{2}|x|^2+c$ at infinity when $c$ satisfies \eqref{eq:c1}, where we used the fact that $\varphi$ is bounded on $\partial\Omega$. By the uniform enclosing sphere condition and Lemma \ref{prop:convex-C1}, $r(\xi)$, $K(\xi)$ and $|p(\xi)|$ are bounded on $\partial\Omega$. By Lemma \ref{lem:A}, $\delta(\xi)$ has a positive lower bound on $\partial\Omega$. Note that $\Omega$ satisfies a enclosing sphere condition on $\partial\Omega$ and thus is convex. Therefore, Proposition \ref{lem:perron0} yields that there exists a constant $c_0$, such that for every $c>c_0$, $u\in C^0(\mathbb{R}^n\setminus\Omega)$ and $u=\varphi$ on $\partial\Omega$. Consequently, \eqref{eq:DiriProb} has a viscosity solution for $c\geq c_1$ with $c_1$ sufficiently large.

We can find the sharp $c_*$ almost without change as in \cite[Theorem 1.2]{Li-Lu}, so we briefly sketch the rest of the proof and omit the details.

\textbf{Step 2.} There is a constant $c_2$ such that for $c<c_2$, there is no viscosity subsolution $\underline{u}$ of \eqref{eq0:MA=1} in $\mathbb{R}^n\setminus\overline{\Omega}$ satisfying
\begin{equation}\label{eq:DiriProb-1}
  \begin{cases}
    \underline{u}=\varphi\quad \mbox{on } \partial\Omega, \\
    \lim_{|x|\to\infty}\left(\underline{u}(x)-\left(\frac{1}{2}|x|^2+c\right)\right)=0.
  \end{cases}
\end{equation}
Precisely, here $c_2$ depends only on $n$, the diameter of $\Omega$ and $\inf_{\partial\Omega}\varphi$, but not depends on the $C^2$ regularity of $\Omega$ and $\varphi$.

\textbf{Step 3.} If \eqref{eq:DiriProb} has a viscosity solution $u_{c_3}$ with $c=c_3$ $(c_3<c_1)$, then \eqref{eq:DiriProb} has a viscosity solution for all $c_4\in(c_3,c_1)$. When proving that there is a viscosity solution of
\begin{equation*}
  \begin{cases}
    \det(D^2u)=1 & \mbox{in } \mathbb{R}^n\setminus\overline{\Omega}, \\
    u=\varphi & \mbox{on } \partial\Omega \\
  \end{cases}
\end{equation*}
provided that there is a viscosity subsolution $\underline{u}_{c_4}$ of \eqref{eq0:MA=1} in $\mathbb{R}^n\setminus\overline{\Omega}$ satisfying \eqref{eq:DiriProb-1} with $c=c_4$, the proof is slightly different from \cite{Li-Lu} due to the weaker condition on $\Omega$. Indeed, let $u$ be as in \eqref{eq:perron-sol2} with $c=c_4$.
Then $\mathcal{S}_{c_4}^\varphi$ is nonempty and $u\leq u_{c_3}$.
Then $u$ satisfies the equation due to Steps 2-3 of the proof of Proposition \ref{lem:det=1}. Fix $R>0$ such that $\Omega\subset\subset B_{R}(0)$. Recalling that $\varphi\in C^0(\partial\Omega)$,  $\Omega$ is convex and $B_{R}(0)\setminus\overline{\Omega}$ satisfies the exterior cone condition. By Proposition \hyperlink{B}{B}, there exists $h\in C^0(\overline{B_{R}(0)\setminus\Omega})$ satisfying
\begin{equation*}
  \begin{cases}
    \Delta h=0 & \mbox{in } B_{R}(0)\setminus\overline{\Omega}, \\
    h=\varphi & \mbox{on } \partial\Omega, \\
    h=\max_{\partial B_{R}(0)}u & \mbox{on }\partial B_{R}(0).
  \end{cases}
\end{equation*}
The comparison principle implies that $u\leq h$ in $\overline{B_{R}(0)\setminus\Omega}$. It follows that
$$\varphi(x_0)=\underline{u}_{c_4}(x_0)\leq\lim_{\mathbb{R}^n\setminus\overline{\Omega}\ni x\to x_0}u(x)\leq h(x_0)=\varphi(x_0).$$

\textbf{Step 4.} The sharp constant $c_*$ is determined by
$$c_*=\inf\{c\in\mathbb{R}|\,\eqref{eq:DiriProb}~\text{has~a~viscosity~solution}\}.$$

\end{proof}

Now, we give the proof of Theorem \ref{thm:main}.

\begin{proof}[Proof of Theorem \ref{thm:main}]
For $A\in\mathcal{A}$, there exists a $n\times n$ orthogonal matrix $P_1$ such that $A=P_1P_2P_2^TP_1^T$, where $P_2=\text{diag}(\sqrt{\lambda_1(A)},\cdots,\sqrt{\lambda_n(A)})$ and $\lambda_1(A), \cdots, \lambda_n(A)$ are the eigenvalues of $A$. Denote $Q=P_1P_2$. Let
$$\widetilde{x}=xQ,\quad\widetilde{\Omega}=\{xQ|\,x\in\Omega\},$$
and
$$\widetilde{\varphi}(\widetilde{x})=\varphi(x)-b\cdot x=\varphi(\widetilde{x}Q^{-1})-b\cdot(\widetilde{x}Q^{-1}).$$
If $\widetilde{\Omega}$ and $\widetilde{\varphi}$ satisfy the assumptions in Proposition \ref{lem:main}, then we conclude that there is $c_*$ depending only on $n$, $\widetilde{\Omega}$, $\widetilde{\varphi}$ such that there a viscosity solution $\widetilde{u}\in C^0(\mathbb{R}^n\setminus\widetilde{\Omega})$ of
\begin{equation}\label{eq:DiriProb3}
\begin{cases}
  \det(D^2\widetilde{u})=1\quad\text{in}~\mathbb{R}^n\setminus\overline{\widetilde{\Omega}},\\
  \widetilde{u}=\widetilde{\varphi}\quad\text{on}~\partial\widetilde{\Omega},\\
  \lim_{|\widetilde x|\to\infty}\Big(\widetilde{u}(\widetilde{x})-\Big(\frac{1}{2}|\widetilde{x}|^2+c\Big)\Big)=0
\end{cases}
\end{equation}
if and only if $c\geq c_*$. Let
$$u(x)=\widetilde{u}(\widetilde{x})+b\cdot x=\widetilde{u}(xQ)+b\cdot x.$$
Then direct calculation shows that $u$ is a viscosity solution of \eqref{eq:DiriProb} in Theorem \ref{thm:main} when $c\geq c_*$. While for $c<c_*$, if \eqref{eq:DiriProb} has a viscosity solution $u$, then
$$\widetilde{u}(\widetilde{x})=u(x)-b\cdot x$$
is a viscosity solution of \eqref{eq:DiriProb3}, which is a contradiction! Hence, we establish Theorem \ref{thm:main}.

It remains to prove that $\widetilde{\Omega}$ satisfies a uniform enclosing sphere condition, $\partial\widetilde{\Omega}\in C^1$, and $\widetilde{\varphi}$ is semi-convex with respect to $\partial\widetilde{\Omega}$. Since $\partial\Omega$ is $C^1$, so is $\partial\widetilde{\Omega}$. Denote by $\partial B_{r(\xi)}(y(\xi))$ an enclosing sphere of $\Omega$ at $\xi$ and $r=\max_{\xi\in\partial\Omega}r(\xi)$. Then $\partial B_r(y(\xi))$ is also an enclosing sphere of $\Omega$ at $\xi$. Denote
$$E_\xi=\{xQ|\,x\in B_r(y(\xi))\}.$$
Then $E_\xi$ is an ellipsoid and
$$\xi Q\in\partial\widetilde{\Omega}\cap\partial E_{\xi}\quad\text{and}\quad\widetilde{\Omega}\subset E_\xi.$$
Thus $\widetilde{\Omega}$ satisfies an enclosing sphere condition with a uniform radius. Hence, $\widetilde{\Omega}$ satisfies a uniform enclosing sphere condition.

We continue to prove that $\widetilde\varphi$ is semi-convex with respect to $\partial\widetilde\Omega$. Fix $\xi\in\partial\Omega$ and denote $\widetilde{\xi}=\xi Q$. Without losing the generality, we may assume the coordinate system in $\mathbb{R}^n$ is just the local coordinate system at $\xi\in\partial\Omega$. $\partial\Omega$ can be locally represented by the graph of
$$x_n=\rho(x')\quad\text{for}~|x'|<\delta(\xi),$$
for some $\delta(\xi)>0$. Let
$\psi(x')=\varphi(x',\rho(x'))$. Since $\varphi$ is semi-convex with respect to $\partial\Omega$ at $\xi$, there exists $K(\xi)>0$ such that for $x',y'\in B'_{\delta(\xi)}(0')$ and $0<t<1$,
\begin{equation}\label{eq:semi}
\begin{split}
& \psi(tx'+(1-t)y')+\frac{K(\xi)}{2}|tx'+(1-t)y'|^2 \\
\leq & t\bigg(\psi(x')+\frac{K(\xi)}{2}|x'|^2\bigg)+(1-t)\bigg(\psi(y')+\frac{K(\xi)}{2}|y'|^2\bigg).
\end{split}
\end{equation}
Suppose that $\partial\widetilde{\Omega}$ can be locally represented by the graph of
$$\widetilde{x}_n=\widetilde{\rho}(\widetilde{x}')\quad\text{for}~|\widetilde{x}'|<\widetilde{\delta}(\widetilde{\xi}),$$
for some $\widetilde{\delta}(\widetilde{\xi})>0$. We may assume
$$\widetilde{\delta}(\widetilde{\xi})\leq\sqrt{\frac{\min_{1\leq i\leq n}\lambda_i(A)}{1+C_0^2}}\delta(\xi),$$
where $C_0$ is the Lipschitz norm of $\partial\widetilde{\Omega}$ depending only on $\partial\Omega$ and $A$. Then for $\widetilde{x}=(\widetilde{x}',\widetilde{\rho}(\widetilde{x}'))\in\partial\widetilde{\Omega}$ with $|\widetilde{x}'|<\widetilde{\delta}(\widetilde{\xi})$, we have for $x=\widetilde{x}Q^{-1}$,
\begin{equation}\label{eq:qq}
\begin{split}
|x'|& \leq|\widetilde{x}Q^{-1}|\leq\sqrt{\max_{1\leq i\leq n}\lambda_i(Q^{-1}(Q^{-1})^T)}|\widetilde{x}|\\
&\leq\sqrt{\max_{1\leq i\leq n}\lambda_i(A^{-1})}\sqrt{1+C_0^2}|\widetilde{x}'|<\delta(\xi).
\end{split}
\end{equation}
Thus we get for $\widetilde{x}=(\widetilde{x}',\widetilde{\rho}(\widetilde{x}'))\in\partial\widetilde{\Omega}$ with $|\widetilde{x}'|<\widetilde{\delta}(\widetilde{\xi})$,
\begin{equation*}
\begin{split}
\widetilde{\psi}(\widetilde{x}'):&=\widetilde{\varphi}(\widetilde{x}',\widetilde{\rho}(\widetilde{x}'))=\widetilde{\varphi}(\widetilde{x})=\varphi(x)-b\cdot x \\
&=\varphi((\widetilde{x}Q^{-1})',\rho((\widetilde{x}Q^{-1})'))-b\cdot(\widetilde{x}Q^{-1}) \\
&=\psi((\widetilde{x}Q^{-1})')-b\cdot(\widetilde{x}Q^{-1}).
\end{split}
\end{equation*}
Replacing $x'$ and $y'$ in \eqref{eq:semi} by $(\widetilde{x}Q^{-1})'$ and $(\widetilde{y}Q^{-1})'$ respectively, together with \eqref{eq:qq}, we obtain
\begin{equation*}
\begin{split}
&\, \widetilde{\psi}\left(t\widetilde{x}'+(1-t)\widetilde{y}'\right) \\
= &\, \psi((t(\widetilde{x}Q^{-1})'+(1-t)(\widetilde{y}Q^{-1})')-b\cdot(t\widetilde{x}Q^{-1}+(1-t)\widetilde{y}Q^{-1}) \\
\leq &\,t\bigg(\psi((\widetilde{x}Q^{-1})')+\frac{K(\xi)}{2}|(\widetilde{x}Q^{-1})'|^2\bigg) +(1-t)\bigg(\psi((\widetilde{y}Q^{-1})')+\frac{K(\xi)}{2}|(\widetilde{y}Q^{-1})'|^2\bigg) \\
&-b\cdot(t\widetilde{x}Q^{-1}+(1-t)\widetilde{y}Q^{-1})-\frac{K(\xi)}{2}\left|t(\widetilde{x}Q^{-1})'+(1-t)(\widetilde{y}Q^{-1})'\right|^2 \\
= &\,t(\psi((\widetilde{x}Q^{-1})')-b\cdot(\widetilde{x}Q^{-1}))+(1-t)(\psi(\widetilde{y}Q^{-1})'-b\cdot (\widetilde{y}Q^{-1})) \\
&\,+\frac{t(1-t)K(\xi)}{2}\left|(\widetilde{x}Q^{-1})'-(\widetilde{y}Q^{-1})'\right|^2 \\
= &\,t\widetilde{\psi}(\widetilde{x}')+(1-t)\widetilde{\psi}(\widetilde{y}')+\frac{t(1-t)K(\xi)}{2}\left|((\widetilde{x}-\widetilde{y})Q^{-1})'\right|^2.
\end{split}
\end{equation*}
Taking
$$\widetilde{K}(\widetilde{\xi})=\frac{1+C_0^2}{\min_{1\leq i\leq n}\lambda_i(A)}K(\xi),$$ we get
\begin{equation*}
\begin{split}
&\, \widetilde{\psi}\left(t\widetilde{x}'+(1-t)\widetilde{y}'\right)+\frac{\widetilde{K}(\widetilde{\xi})}{2}\left|t\widetilde{x}'+(1-t)\widetilde{y}'\right|^2 \\
\leq &\,t\bigg(\widetilde{\psi}(\widetilde{x}')+\frac{\widetilde{K}(\widetilde{\xi})}{2}|\widetilde{x}'|^2\bigg) +(1-t)\bigg(\widetilde{\psi}(\widetilde{y}')+\frac{\widetilde{K}(\widetilde{\xi})}{2}|\widetilde{y}'|^2\bigg) \\
&\, +\frac{t(1-t)K(\xi)}{2}\left|((\widetilde{x}-\widetilde{y})Q^{-1})'\right|^2-\frac{t(1-t)\widetilde{K}(\widetilde{\xi})}{2}\left|\widetilde{x}'-\widetilde{y}'\right|^2 \\
\leq &\,t\bigg(\widetilde{\psi}(\widetilde{x}')+\frac{\widetilde{K}(\widetilde{\xi})}{2}|\widetilde{x}'|^2\bigg) +(1-t)\bigg(\widetilde{\psi}(\widetilde{y}')+\frac{\widetilde{K}(\widetilde{\xi})}{2}|\widetilde{y}'|^2\bigg),
\end{split}
\end{equation*}
where we used
$$|((\widetilde{x}-\widetilde{y})Q^{-1})'|\leq\sqrt{\max_{1\leq i\leq n}\lambda_i(A^{-1})}\sqrt{1+C_0^2}|\widetilde{x}'-\widetilde{y}'|$$
in the last ``$\leq$'' as in \eqref{eq:qq}. That is, $\widetilde{\varphi}$ is semi-convex with respect to $\partial\widetilde{\Omega}$ at $\widetilde{\xi}$.
This finishes the proof of Theorem \ref{thm:main}.
\end{proof}

\begin{remark}
From the proof of Proposition \ref{lem:main} and Theorem \ref{thm:main}, we see that Theorem \ref{thm:main} actually holds as long as $\Omega$ satisfies \eqref{H} and a uniform enclosing sphere condition. We note that there exist bounded and convex domains but not satisfy \eqref{H}; see Example 3 in Appendix.
\end{remark}

\section*{Appendix}
Here we prove some conclusions that are involved in the main context. We also include some examples to demonstrate that the conditions in Theorem \ref{thm:main} holds for more boundary values and domains than that in Theorem \ref{thm:CL2003}. Meanwhile, we give some examples to further understand the conditions in Theorem \ref{thm:main}.

\begin{propA}\label{prop:equivalence}
 \hypertarget{A}{Let} $\Omega$ be a domain of $\mathbb{R}^n$, $\partial\Omega\in C^2$. If $\Omega$ satisfies a uniform enclosing sphere condition, then $\Omega$ is bounded and strictly convex. The converse is also true.
\end{propA}

\begin{proof}
If $\Omega$ satisfies a uniform enclosing sphere condition, then $\Omega$ is clearly bounded. Under the local coordinate system at $\xi\in\partial\Omega$, there exists an enclosing sphere $\partial B_r(y)$ at $\xi$ with $y=(0',r)$. It is clear from $\partial\Omega\subset\overline{B_r(y)}$ that
  $$|x'|^2+(\rho(x')-r)^2\leq r^2.$$
  This gives
  \begin{equation*}
  \rho(x')\geq\frac{1}{2r}|x'|^2.
  \end{equation*}
  Since $\rho(0')=0$ and $D'\rho(0')=0'$, we have $D'^2\rho(0')>0$. Hence, $\Omega$ is strictly convex.

  Conversely, if $\Omega$ is bounded and strictly convex at $\xi\in\partial\Omega$, there exist constants $0<\varepsilon,\delta<1$ independent of $\xi$, such that
  $$D'^2\rho(0')\geq\varepsilon I_{n-1}\quad\text{and}\quad\rho(x')\geq\frac{\varepsilon}{2}|x'|^2\text{ for }|x'|<\delta,$$
  under the local coordinate system at $\xi$, where $I_{n-1}$ is the $(n-1)\times(n-1)$ identity matrix. Take
  $$r=\max\bigg\{\frac{2}{\varepsilon},\frac{(\text{diam}\Omega+1)^2}{\varepsilon\delta^2}+\frac{\varepsilon}{4}\delta^2\bigg\}\quad\text{and}\quad y=(0',r).$$
  We claim that for any $z\in\partial\Omega$, $z\in\overline{B_r(y)}$. Indeed,
  if $z=(z',\rho(z'))$ with $|z'|<\delta$, then
  \begin{equation*}
  \begin{split}
  |z-y|^2 & =|z'|^2+(r-\rho(z'))^2\leq|z'|^2+\left(r-\frac{\varepsilon}{2}|z'|^2\right)^2 \\
  & =r^2+(1-\varepsilon r)|z'|^2+\frac{\varepsilon^2}{4}|z'|^4 \\
  &\leq r^2.
  \end{split}
  \end{equation*}
  If $z\in\partial\Omega\setminus\{(z',\rho(z'))|\,|z'|<\delta\}$, then $\frac{\varepsilon}{2}\delta^2\leq z_n\leq r$. It follows that
  \begin{equation*}
    \begin{split}
       |z-y|^2 & =|z'|^2+(r-z_n)^2\leq|z'|^2+\left(r-\frac{\varepsilon}{2}\delta^2\right)^2\\
         & = r^2+(\text{diam}\Omega)^2-\varepsilon\delta^2r+\frac{\varepsilon^2}{4}\delta^4 \\
         & \leq r^2.
    \end{split}
  \end{equation*}
  Hence, $\Omega$ satisfies a uniform enclosing sphere condition.
\end{proof}

\begin{propB}\label{lem:cone}
\hypertarget{B}{Let} $\Omega$ be a bounded domain of $\mathbb{R}^n$ satisfying an exterior cone condition, $n\geq2$; that is for every $\xi\in\partial\Omega$, there exists a finite right circular cone $\mathcal{C}$, with vertex $\xi$, such that $\overline{\Omega}\cap\overline{\mathcal{C}}=\{\xi\}$. Let $\varphi\in C^0(\partial\Omega)$. Then there exists a solution $u\in C^0(\overline{\Omega})$ of
\begin{equation*}
  \begin{cases}
    \Delta u=0 & \mbox{in } \Omega, \\
    u=\varphi & \mbox{on } \partial\Omega.
  \end{cases}
\end{equation*}
\end{propB}

\begin{proof}
The proposition was mentioned in \cite[Problem 2.12]{GT}, and here we give the proof for the reader's convenience. By \cite[Theorem 2.14]{GT}, it suffices to prove that for every $\xi\in\partial\Omega$, there is a local barrier $w$ at $\xi\in\partial\Omega$ relative to $\Omega$. Without losing the generality, we may assume $\xi=0$ and $x_n$-axis is in the direction of the axis of $\mathcal{C}$. For $x\neq0$, let
$$r(x)=|x|\quad\text{and}\quad\theta(x)=\arccos\frac{x_n}{r}.$$
Take $0<r_0<1$ and $\theta_0\in(\frac{\pi}{2},\pi)$ such that
$$\mathcal{C}:=\{0\}\cup\{x\in\mathbb{R}^n|\,0<r(x)\leq r_0,~\theta_0\leq\theta(x)\leq\pi\},$$
and
$$\Omega\cap B_{r_0}(0)\subset\{x\in B_{r_0}(0)|\,0\leq\theta(x)<\theta_0\}.$$
Consider $w$ given by $w(0)=0$ and
$$w(x)=r^\lambda f(\theta)\quad\text{for}~x\neq0,$$
where the constant $\lambda$ and the function $f$ will be chosen later. For $n=2$, we can take $0<\lambda\leq\frac{\pi}{2\theta_0}$ and $f(\theta)=\cos(\lambda\theta)$. Then
$$\Delta w=r^{\lambda-2}(\lambda^2f(\theta)+f''(\theta))=0\quad\text{in}~\Omega\cap B_{r_0}(0),$$
and $w>w(0)$ in $\Omega\cap B_{r_0}(0)$. For $n\geq3$, direct calculation gives
$$\theta_i(x)=\frac{-1}{\sqrt{1-(\frac{x_n}{r})^2}}\bigg(\frac{\delta_{in}}{r}-\frac{x_n}{r^2}\frac{x_i}{r}\bigg)=-r^{-1}\bigg(\delta_{in}-\frac{x_ix_n}{r^2}\bigg)\csc\theta,$$
and
$$w_{i}(x)=\lambda r^{\lambda-2}x_if(\theta)-r^{\lambda-1}\bigg(\delta_{in}-\frac{x_ix_n}{r^2}\bigg)f'(\theta)\csc\theta,$$
\begin{equation*}
\begin{split}
w_{ii}(x)= & \bigg(\lambda r^{\lambda-2}+\lambda(\lambda-2)r^{\lambda-3}\frac{x_i}{r}x_i\bigg)f(\theta)-\lambda r^{\lambda-3}\bigg(\delta_{in}-\frac{x_ix_n}{r^2}\bigg)x_if'(\theta)\csc\theta\\
&-\bigg((\lambda-1)r^{\lambda-2}\frac{x_i}{r}\bigg(\delta_{in}-\frac{x_ix_n}{r^2}\bigg)-r^{\lambda-1}
\bigg(\frac{\delta_{in}x_i+x_n}{r^2}-2\frac{x_ix_n}{r^3}\frac{x_i}{r}\bigg)\bigg) f'(\theta)\csc\theta\\
&-r^{\lambda-2}\bigg(\delta_{in}-\frac{x_ix_n}{r^2}\bigg)^2f'(\theta)(\csc\theta)^2\cot\theta\\
&+r^{\lambda-2}\bigg(\delta_{in}-\frac{x_ix_n}{r^2}\bigg)^2f''(\theta)(\csc\theta)^2.
\end{split}
\end{equation*}
From this, we have
\begin{equation*}
\begin{split}
\Delta w & =r^{\lambda-2}\bigg(\lambda(\lambda+n-2)f(\theta)+(n-2)\cot\theta f'(\theta)+f''(\theta)\bigg)\\
& =r^{\lambda-2}\bigg(\lambda(\lambda+n-2)f(\theta)+(\csc\theta)^{n-2}\frac{\mathrm{d}}{\mathrm{d}\theta}\bigg((\sin\theta)^{n-2}f'(\theta)\bigg)\bigg).
\end{split}
\end{equation*}
Fix $\theta_0<\Theta_0<\pi$, let
$$f(\theta)=\int_{\theta}^{\Theta_0}(t\csc t)^{n-2}\mathrm{d}t.$$
Then
\begin{equation*}
\begin{split}
\Delta w & =r^{\lambda-2}\bigg(\lambda(\lambda+n-2)f(\theta)-(n-2)\theta^{n-3}(\csc\theta)^{n-2}\bigg)\\
& \leq r^{\lambda-2}\bigg(\lambda(\lambda+n-2)f(\theta)-\frac{n-2}{\theta_0}\bigg)\quad\text{in}~\Omega\cap B_{r_0}(0).
\end{split}
\end{equation*}
Take
$$\lambda=\frac{n-2}{2}\bigg(\bigg(1+\frac{4}{(n-2)\theta_0f(0)}\bigg)^\frac{1}{2}-1\bigg)>0.$$
Combining $f(\theta)\leq f(0)$, we get
$$\Delta w\leq0\quad\text{in}~\Omega\cap B_{r_0}(0).$$
Clearly, $w>w(0)$ in $\overline{\Omega\cap B_{r_0}(0)}\setminus\{0\}$. Therefore, $w$ is a local barrier at $0$.
\end{proof}

We end this section by giving a few specific examples.
\begin{example}
  Let $\Omega$ be a convex domain with $\partial\Omega\in C^2$ and $0\in\partial\Omega$. Then $\varphi(x)=|x|$ is semi-convex with respect to $\partial\Omega$ but $\varphi\notin C^2(\partial\Omega)$. The definition of a $C^2$ function on $\partial\Omega\in C^2$ can refer to \cite[Chapter 6.2]{GT}.

  We first check that $\varphi=|x|$ is not $C^2$ at $0$. Indeed, under the local coordinate system at $0$, we have for $i=1,\cdots,n-1$ and $x'=(0,\cdots,0,x_i,0,\cdots,0)\in\mathbb{R}^{n-1}$,
  $$\lim_{x_i\to0^+}\frac{\varphi(x',\rho(x'))-\varphi(0',\rho(0'))}{x_i}=\lim_{x_i\to0^+}\frac{\sqrt{x_i^2+\rho(x')^2}}{x_i}
  \geq1,$$
  $$\lim_{x_i\to0^-}\frac{\varphi(x',\rho(x'))-\varphi(0',\rho(0'))}{x_i}=\lim_{x_i\to0^-}\frac{\sqrt{x_i^2+\rho(x')^2}}{x_i}
  \leq-1.$$
  Hence, $\varphi$ is not $C^1$ at $0$.

  On the other hand, under the local coordinate system at $0$, $\rho\geq0$ is convex due to the convexity of $\Omega$. Then
  for $\psi(x'):=\varphi(x',\rho(x'))$ and $x,y\in\partial\Omega$ near $0$, we have
  \begin{equation*}
  \begin{split}
  (\psi(tx'+(1-t)y'))^2 =&\,  (\varphi(tx'+(1-t)y',\rho(tx'+(1-t)y')))^2 \\
  = &\, |tx'+(1-t)y'|^2+(\rho(tx'+(1-t)y'))^2 \\
  \leq &\, |tx'+(1-t)y'|^2+(t\rho(x')+(1-t)\rho(y'))^2 \\
  = &\, |tx+(1-t)y|^2 \\
  \leq &\, (t|x|+(1-t)|y|)^2 \\
  = &\, (t\psi(x')+(1-t)\psi(y'))^2.
  \end{split}
  \end{equation*}
  Together with $\psi\geq0$, we get
  $$\psi(tx'+(1-t)y')\leq t\psi(x')+(1-t)\psi(y'),$$
  and so $\psi$ is convex near $0'$. Hence, $\varphi$ is semi-convex with respect to $\partial\Omega$ at $0$. For $\xi\in\partial\Omega\setminus\{0\}$, $\varphi$ is $C^2$ at $\xi$. It is obvious that $\varphi$ is semi-convex with respect to $\partial\Omega$ at $\xi$.
\end{example}

%


\begin{example}
There exist domains $\Omega$ satisfying that there is an enclosing sphere at every $\xi\in\partial\Omega$, but $r(\xi)$ is not bounded on $\partial\Omega$. For instance,
$$\Omega=\{x\in\mathbb{R}^n|~x_1,\cdots,x_{n-1}>0,|x'|^3<x_n<|x'|^{1/3}\}.$$

\begin{figure}[h]
  \centering
  \includegraphics[width=6cm]{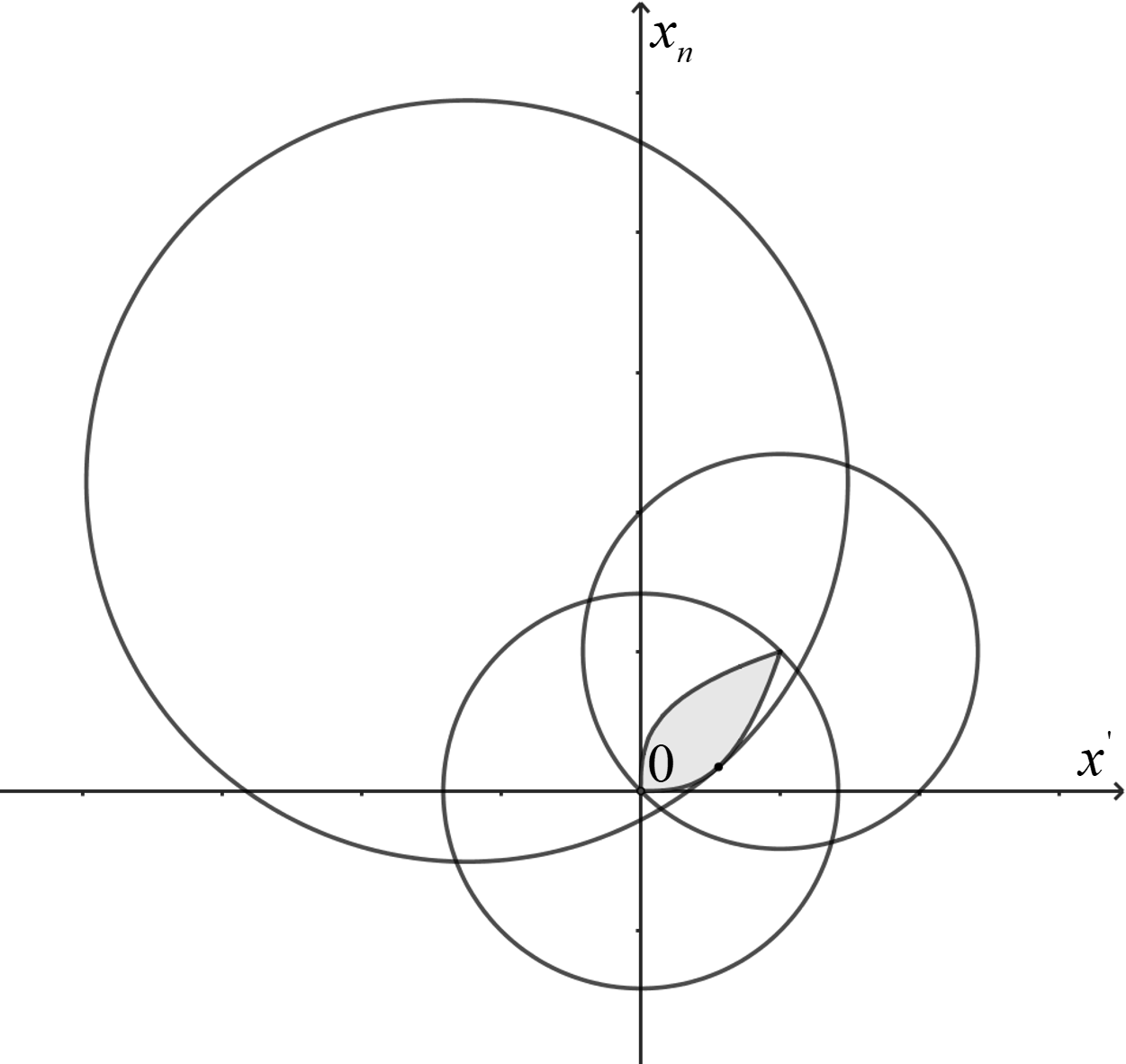}
\end{figure}

\noindent Indeed, an enclosing sphere at $\xi=(0,\cdots,0)$ can be
$$|x-(1,\cdots,1)|^2=n.$$
An enclosing sphere at $\xi=(1,\cdots,1)$ can be
$$|x|^2=n.$$
Since $\partial\Omega$ is $C^2$ and strictly convex except for points $(0,\cdots,0)$ and $(1,\cdots,1)$, there is an enclosing sphere at these boundary points due to the proof of Proposition \hyperlink{A}{A}.

On the other hand, we will prove that the radius $r$ of the enclosing sphere is not bounded on $\partial\Omega$. Denote $\Gamma=\partial\Omega\cap\{\xi\in\mathbb{R}^n|\,0<|\xi'|<1,~\xi_n=|\xi'|^3\}$. Direct calculation gives that the unit inner normal vector $\nu$ of $\partial\Omega$ at $\xi=(\xi',|\xi'|^3)\in\Gamma$ is
$$\nu=\frac{(-3|\xi'|\xi',1)}{\sqrt{9|\xi'|^4+1}}.$$
Denote by $\partial B_r(y)$ an enclosing sphere at $\xi\in\Gamma$. Then $y=\xi+r\nu$, and an enclosing sphere at $\xi$ is
$$|x-\xi-r\nu|=r.$$
Since $0\in\partial\Omega\subset\overline{B_r(y)}$, we have
$$|-\xi-r\nu|\leq r.$$
That is $|\xi|^2+2r\nu\cdot\xi\leq0$. It follows that
\begin{equation*}
   r\geq\frac{|\xi|^2}{-2\nu\cdot\xi}=\frac{(1+|\xi'|^4)\sqrt{9|\xi^4|+1}}{4|\xi'|}\to\infty\quad\text{as}~\xi\to0.
\end{equation*}


\end{example}

\begin{example}
There exist many domains $\Omega$ which satisfy \eqref{H} and a uniform enclosing sphere condition but $\partial\Omega\notin C^1$. For instance $\Omega=\Omega_1\cap\Omega_2$, where
  $$\Omega_1=\{x\in\mathbb{R}^n|\,|x|<1\}\quad\text{and}\quad\Omega_2=\{x\in\mathbb{R}^n|\,|x'|^2+(x_n-a)^2<1\},$$
and $|a|<\sqrt{2}$. Indeed, $\partial\Omega$ is clearly not $C^2$ on $\partial\Omega_1\cap\partial\Omega_2$. On the other hand, $\Omega$ clearly satisfies \eqref{H} and a uniform enclosing sphere condition.

In addition,  $\Omega$ does not satisfy \eqref{H} if $|a|\geq\sqrt{2}$.
\end{example}

\section*{Acknowledgements}
The author C. Wang would like to thank Professor Bo Wang for his helpful suggestions in preliminary discussions.

\end{document}